\documentclass[12pt]{amsart}
\usepackage{latexsym}
\usepackage{amssymb,amsfonts,amsmath,mathrsfs}
\usepackage{geometry}
\newtheorem{theorem}{Theorem}[section]

\newtheorem{proposition}[theorem]{Proposition}

\newtheorem{lemma}[theorem]{Lemma}

\newtheorem*{theorem1.2}{Theorem 1.2}
\newtheorem*{theorem1.3}{Theorem 1.3}
\newcommand{\tmop}[1]{\ensuremath{\operatorname{#1}}}
\numberwithin{equation}{section}

\begin{document}

\title[Zeros of modular forms and QUE]{Small scale distribution of zeros and mass of modular forms}

\author[S. Lester]{Stephen Lester}
\address{School of Mathematical Sciences, Tel Aviv University, Tel Aviv 69978, Israel}
\email{slester@post.tau.ac.il}

\author[K. Matom\"aki]{Kaisa Matom\"aki}
\address{Department of Mathematics and Statistics, University of Turku,
20014 Turku, Finland}
\email{ksmato@utu.fi}

\author[M. Radziwi\l \l]{Maksym Radziwi\l \l}
\address{Centre de Recherches Mathematiques Universite de Montreal P. O Box 6128 \\
Centre-Ville Station Montreal \\ Quebec H3C 3J7 }
\curraddr{Department of Mathematics, Rutgers University \\ Hill Center for the Mathematical Sciences \\
110 Frelinghuysen Rd., Piscataway, NJ 08854-8019 }
\email{radziwil@crm.umontreal.edu}

\thanks{The second author was supported by the Academy of Finland grants no. 137883 and 138522. The research leading to these results has received funding
from the European Research Council under the European Union's
Seventh Framework Programme (FP7/2007-2013) / ERC grant agreement
n$^{\text{o}}$ 320755.}

\begin{abstract}
We study the behavior of zeros and mass of holomorphic Hecke cusp forms on $\tmop{SL}_2(\mathbb Z) \backslash \mathbb H$ at small scales. 
In particular, we examine the distribution of the zeros within hyperbolic balls whose radii shrink sufficiently slowly as $k \rightarrow \infty$. We show that the zeros equidistribute within such balls as $k \rightarrow \infty$ as long as the radii shrink at a rate at most a small power of $1/\log k$.  This relies on a new, effective, proof of Rudnick's theorem on equidistribution of the zeros and on an effective version of Quantum Unique Ergodicity for holomorphic forms, which we obtain in this paper. 

We also examine the distribution of the zeros near the cusp of $\tmop{SL}_2(\mathbb Z) \backslash \mathbb H$.
Ghosh and Sarnak conjectured that almost all the zeros here lie on two vertical geodesics. We show that for almost all forms a positive proportion of zeros high in the cusp do lie on these geodesics. For all forms, we assume the Generalized Lindel\"of Hypothesis and establish a lower bound on the number of zeros that lie on these geodesics, which is significantly stronger
than the previous unconditional results.
\end{abstract}

\maketitle

\section{Introduction}
Let $f$ be a modular form of weight $k$ for $\tmop{SL}_2(\mathbb Z)$. A classical result in
the theory of modular forms states that the  number of properly weighted zeros of $f$ in  $\tmop{SL}_2(\mathbb Z) \backslash \mathbb H$ equals $k/12$. Inside the fundamental domain $\mathcal F=\{ z \in \mathbb H
: -1/2 \le \tmop{Re}(z) < 1/2, |z| \ge 1 \}$ the distribution of the zeros of different modular forms of weight $k$ can vary drastically.
For instance, F.K.C. Rankin and H.P.F. Swinnerton-Dyer \cite{RaS} have proved
that all the zeros of the holomorphic Eisenstein series
\[
E_k(z)=\frac12 \sum_{(c,d)=1} \frac{1}{(cz+d)^k}
\]
that lie inside $\mathcal F$ lie on the arc $\{ |z|=1 \}$. Moreover, the zeros of $E_k(z)$ are
uniformly distributed on this arc as $k \rightarrow \infty$. In contrast, consider powers of the modular discriminant, that is, $\Delta(z)^{\frac{k}{12}}$ with $12 | k$. This function
is a weight $k$ cusp form and has one distinct zero at $\infty$ with multiplicity $k/12$.

The weight $k$ Hecke cusp forms constitute a natural basis for the space of weight $k$ modular forms and 
the distribution of their zeros differs from the previous two examples.
Using methods from potential theory,  Rudnick \cite{R} 
showed that the zeros of Hecke cusp forms equidistribute in the fundamental domain $\mathcal{F}$
with respect to hyperbolic measure in the limit as the weight tends to infinity.
Rudnick's result originally relied on the then unproven Quantum Unique Ergodicity (QUE) conjecture
for holomorphic Hecke cusp forms.
However this is now a theorem proved by Holowinsky 
and Soundararajan \cite{HS} and so Rudnick's result on the equidistribution of zeros holds unconditionally. 

It is natural to study what happens beyond equidistribution, and to investigate
the distribution of zeros and mass of Hecke cusp forms at smaller scales. 
That is, to examine the behavior of the zeros and mass 
within sets whose hyperbolic area tends to zero  at a quantitative rate as the weight $k \rightarrow \infty$.
For the zeros, we consider the following two different aspects of this problem:
\begin{itemize}
\item[1)] The distribution of zeros of Hecke cusp forms within hyperbolic balls $B(z_0,r_k) \subset \mathcal F$
with $r_k \rightarrow 0$ sufficiently slowly as $k \rightarrow \infty$.
\item[2)] The distribution of the zeros of Hecke cusp forms in the domain
\[
\mathcal F_{Y}=\{ z \in \mathcal F : \tmop{Im}(z) >Y \} \qquad Y \ge \sqrt{k \log k}.
\]
\end{itemize}
The second problem also examines the zeros of $f$ at a small scale since the hyperbolic area of $\mathcal F_Y$ equals $1/Y$ and
tends to zero as the weight tends to infinity.
This problem was originally studied by Ghosh and Sarnak \cite{GS} who proved that many of the zeros of $f$ that lie
inside $\mathcal F_{Y}$ lie on each of the vertical geodesics $\tmop{Re}(z)=-1/2$ and $\tmop{Re}(z)=0$.

Additionally, building on the techniques
developed by Holowinsky and Soundararajan we prove an effective form of QUE. Our result
also applies to the small scale setting and we show that the $L^2$-mass of a weight $k$ Hecke cusp 
 equidistributes inside a rectangle whose hyperbolic area shrinks sufficiently slowly as $k \rightarrow \infty$.
This complements recent work of Young \cite{Y} who studied QUE at even smaller scales under the assumption of the Generalized
Lindel\"of Hypothesis. Notably, Young's work also applies to Hecke-Maass forms whereas the analog of our result for Hecke-Maass forms is open.

\subsection{Zeros of Hecke cusp forms in shrinking hyperbolic balls and effective QUE}
Two immediate difficulties appear when attempting
to understand the distribution of zeros of Hecke cusp forms
in shrinking hyperbolic balls: First of all, it is not clear if it is possible to
adapt
Rudnick's argument since it relies on a compactness argument, which is not effective and does not apply
to the small scale setting.
 Secondly, the current results on QUE
do not establish a rate of convergence. We remedy the first difficulty by finding a new
proof of Rudnick's theorem, which is effective.  We address the second difficulty by revisiting
the work of Holowinsky and Soundararajan and extracting a rate of convergence from their result.
This leads to the following theorem.
\begin{theorem} \label{zeros thm1}
Let $f_k$ be a sequence of Hecke cusp forms of weight $k$. Also, let $B(z_0, r) \subset \mathcal{F}$ 
be the hyperbolic ball centered at $z_0$ and of radius $r$, with
$z_0$ fixed and $r \ge (\log k)^{-\delta / 2 + \varepsilon}$ where $\delta = \tfrac{1}{7} \cdot (31/2-4\sqrt{15}) = 0.001152 \ldots$. 
Then as $k \rightarrow \infty$,we have
$$
\frac{\# \{ \varrho_f \in B(z_0, r): f_k(\varrho_f) = 0 \}}{\# \{ \varrho_f \in \mathcal{F}: f_k(\varrho_f) = 0\}} = \frac{3}{\pi} \iint_{B(z_0,r)}
\frac{dx dy}{y^2}  + O \Big ( r (\log k)^{-\delta / 2 + \varepsilon} \Big ).
$$
\end{theorem}

This result is far from optimal since we expect equidistribution for the zeros of Hecke cusp forms nearly all the way down to the Planck scale. That is, the zeros of Hecke modular forms should equidistribute with respect to hyperbolic measure
within hyperbolic balls with area as small as $k^{-1+\varepsilon}$. Assuming the Generalized Lindel\"of Hypothesis
we can show this happens within
hyperbolic balls with area as small as $k^{-1/4+\varepsilon}$.

\begin{theorem} \label{zeros thmlin}
Assume the Generalized Lindel\"of Hypothesis. Let $f_k$ be a sequence of Hecke cusp forms of weight $k$.  Also, let $B(z_0, r) \subset \mathcal{F}$ be the hyperbolic ball centered at $z_0$ and of radius $r$,
with $z_0$ fixed and $r \ge k^{-1/8+\varepsilon}$. Then as $k \rightarrow \infty$ we have
$$
\frac{\# \{ \varrho_f \in B(z_0, r): f_k(\varrho_f) = 0 \}}{\# \{ \varrho_f \in \mathcal{F}: f_k(\varrho_f) = 0\}} = \frac{3}{\pi} \iint_{B(z_0,r)}
\frac{dx dy}{y^2}  + O \Big ( r k^{-1/8+\varepsilon} \Big ) .
$$
\end{theorem}

While QUE establishes that the mass of $y^k |f(z)|^2$ equidistributes as the weight $k$
of $f$ grows, our proof of Theorem 2.1
shows that the equidistribution of the zeros follows from the much weaker
condition: For any 
fixed $\varepsilon > 0$ and for any fixed domain
$\mathcal{R}$, we have
$$
\iint_{\mathcal{R}} y^k \cdot |f(z)|^2 \cdot \frac{dx dy}{y^2} \gg e^{-\varepsilon k}.
$$
We were not able to make use of this weaker condition, but remain hopeful that it will be useful in later works
(see Theorem 2.1 for precise results). 

To understand the mass of $f$ in shrinking sets we obtain 
the following effective version of Quantum Unique Ergodicity in the holomorphic case. 
\begin{theorem}[Effective QUE] \label{effective que}
Let $f$ be a Hecke cusp form of weight $k$. Then,
$$
\sup_{\mathcal{R} \subset \mathcal{F}} \Bigg | 
\iint_{\mathcal{R}} y^k |f(z)|^2 \ \frac{dx dy}{y^2} - \frac{3}{\pi}
\iint_{\mathcal{R}} \frac{dx dy}{y^2} \Bigg | \ll_{\varepsilon}
(\log k)^{-\delta + \varepsilon}
$$
with $\delta = \tfrac{1}{7} \cdot (31/2-4\sqrt{15}) = 0.001152 \ldots$  and 
where the supremum is taken over all the rectangles $\mathcal{R}$ lying inside
the fundamental domain $\mathcal{F}$
that have sides parallel to the coordinate axes. 
\end{theorem} 
For general domains $\mathcal{R}$
we cannot extract from the argument of Holowinsky and Soundararajan a
saving exceeding a small power of $\log k$. However, assuming the Generalized
Lindel\"of Hypothesis, 
Watson \cite{W} 
and Young \cite{Y} have established a power saving bound, which is an important
ingredient in the proof of Theorem \ref{zeros thmlin}.
 On the unconditional front, it was proven by 
 Luo and Sarnak \cite{LS1,LS2} that one can obtain comparable results
on average, obtaining a power saving bound for most forms $f$. Combining this input with our new proof of Rudnick's theorem gives the following variant of Theorem \ref{zeros thm1}. 
\begin{theorem} \label{zeros thm2}
Let $\mathcal{H}_k$ be a Hecke basis for the set of weight $k$ cusp forms. Let $ \delta > 0$. 
Then, for all but at most $\ll k^{20/21 + 4\delta}$ forms $f \in \mathcal H_k$, we have for $r \ge k^{-\delta/2}$
$$
\frac{\#\{\varrho_f \in B(z_0, r): f_k(\varrho_f) = 0\}}{\#\{\varrho_f \in \mathcal{F}: f_k(\varrho_f) = 0\}}
= \frac{3}{\pi} \iint_{B(z_0,r)} \frac{dx dy}{y^2} + O\Big ( r k^{-\delta/2} \log k \Big ).
$$
\end{theorem}

\subsection{Zeros of Hecke cusp forms in shrinking Siegel domains}
We also consider the distribution of 
the zeros of Hecke cusp forms within the set $\mathcal F_Y=\{z \in \mathcal F : \tmop{Im}(z)>Y \}$
with $Y > \sqrt{k \log k}$. The hyperbolic area of $\mathcal F_Y$ equals $\frac{1}{Y}$,
and Ghosh and Sarnak \cite{GS} proved 
for a weight $k$ Hecke cusp form, $f_k$, that
$\frac{k}{Y}
\ll \# \{ \varrho_f \in \mathcal F_Y \} \ll \frac{k}{Y}$.
They also observed that equidistribution should not happen here
and conjectured that almost all the zeros of $f_k$ in $\mathcal F_Y$ lie on the
vertical geodesics $\tmop{Re}(z)=-1/2$ and $\tmop{Re}(z)=0$
with one half lying on each line.

In support of their conjecture Ghosh and Sarnak showed that many of the zeros of 
$f_k$ in $\mathcal F_Y$ lie on
segments of the vertical lines $\tmop{Re}(z)=0$ and $\tmop{Re}(z)=-1/2$. They proved that
\begin{equation}
\label{eq:GS}
\# \{ \varrho_f \in \mathcal F_Y :  \tmop{Re}(\varrho_f)=0 \mbox{ or } \tmop{Re}(\varrho_f)=-1/2 \} \gg (k/Y)^{\frac{1}{2}-\frac{1}{40}-\epsilon}.
\end{equation}
The term $1/40$ in their result was subsequently removed in \cite{Ma} by the second named author.

In support of Ghosh and Sarnak's conjecture, we establish the following result. 
\begin{theorem}
\label{thm:realzeravg}
Let $\varepsilon > 0$ be fixed. There exists a subset $\mathcal{S}_k \subset \mathcal{H}_k$, containing more than $(1 - \varepsilon) |\mathcal{H}_k|$ elements, and such that every $f \in \mathcal{S}_k$ we have
\[
\# \{ \varrho_f \in \mathcal F_Y :  \tmop{Re}(\varrho_f)=0 \}  \geq c(\varepsilon)\cdot \# \{ \varrho_f \in \mathcal F_Y \} 
\]
and
\[
\# \{ \varrho_f \in \mathcal F_Y :  \tmop{Re}(\varrho_f)=-1/2 \}  \geq c(\varepsilon) \cdot \# \{ \varrho_f \in \mathcal F_Y\}
\]
provided that $\delta(\varepsilon) k > Y > \sqrt{k \log k}$ and $k \rightarrow \infty$. The constants $\delta(\epsilon)$ and $c(\varepsilon)$ depend only on $\varepsilon$.
\end{theorem}

 The proof of Theorem \ref{thm:realzeravg} relies on a very recent result 
on multiplicative functions by the second and third author \cite{MaRa}.
For individual forms $f$ we cannot do as well, even on the
assumption of the Lindel\"of or Riemann Hypothesis. The reason is the following: In order to produce sign changes
of $f$ we look at sign changes of the coefficients $\lambda_f(n)$. In order to obtain a positive proportion of the
zeros on the line we need a positive proportion of sign changes between the coefficients of $\lambda_f(n)$, in 
appropriate ranges of $n$. However we cannot have a positive proportion of sign changes if for example, for all primes $p \leq (\log k)^{2 - \varepsilon}$, we have $\lambda_f(p) = 0$. Unfortunately even on the Riemann Hypothesis we cannot currently
rule out this scenario.

Nonetheless on the Lindel\"of Hypothesis we can
still obtain the following result, which is significantly stronger than the previous unconditional result. 
\begin{theorem} \label{thm:signs}
Assume the Generalized Lindel\"of Hypothesis. Then
for any $\varepsilon>0$
\begin{equation} \label{first part}
\# \{ \varrho_f \in \mathcal F_Y :  \tmop{Re}(\varrho_f)=0 \}  \gg (k/Y)^{1-\varepsilon} 
\end{equation}
and
\begin{equation} \label{second part}
\# \{ \varrho_f \in \mathcal F_Y :  \tmop{Re}(\varrho_f)=-1/2 \} \gg (k/Y)^{1-\varepsilon},
\end{equation}
provided that $\sqrt{k \log k} < Y < k^{1-\delta}$ for some $\delta>0$.
\end{theorem}

The paper is organized as follows: In Section 2 we investigate the results related to equidistribution in shrinking sets.
In Section 3 we prove the results on zeros high in the cusp. Finally in Section 4 we establish the effective version 
of Quantum Unique Ergodicity. 

\section{Zeros of cusp forms in shrinking geodesic balls}
Let $\phi$ be a smooth function that
is compactly supported within $\mathcal F$.
Also, let $D_r(z)$ be the disk of radius $r$ centered at $z$.
Given a cusp form $f$, and a compact subset $\mathcal{R} \subset \mathcal{F}$, define,
$$
\mu_f(\mathcal{R}) := \iint_{\mathcal{R}} y^k |f(z)|^2 \frac{dx dy}{y^2}. 
$$
Here the form $f$ is assumed to be normalized so that $\mu_f(\mathcal {F}) = 1$.  
Also, let $\Delta=-y^2\Big(\frac{\partial^2}{\partial x^2}+\frac{\partial^2}{\partial y^2}\Big)$
denote the hyperbolic Laplacian.
The main component of the proofs
of Theorems \ref{zeros thm1}, \ref{zeros thmlin} and \ref{zeros thm2}  is the following:
\begin{theorem} \label{equi prop}
Let $f$ be a Hecke cusp form, and, $\mathcal{R} \subset \{ z\in \mathcal F : \tmop{Im}(z) \le B\}$
where $B>1$. Also, let $h(k) > \frac{\log k}{k}$
and $\phi$
be a smooth compactly supported function in $\mathcal{R}$ such
that $\Delta \phi \ll h(k)^{-A}$ for some $A \ge 0$. Suppose for every
 $z_0 \in \mathcal{R}$ 
and $k \ge K(B)$ there exists a point $z_1=x_1+iy_1 \in D_{h(k)}(z_0)$ satisfying
\begin{equation} \label{hypothesis}
y_1^k|f(z_1)|^2 \gg e^{-k h(k)}.
\end{equation}
Then,
\begin{equation} \label{main eq}
\begin{split}
\sum_{\varrho_f }\phi(\varrho_f)=&
\frac{k}{12} \cdot \frac{3}{\pi}  \iint_{\mathcal F}
\phi(z) \frac{dx dy}{y^2}+O_B(k \cdot h(k)^2)\\
&+O_{A, B} \Big(k \cdot h(k) \log 1/h(k)  \iint_{\mathcal F}
|\Delta \phi(z) | \frac{dx \, dy}{y^2} \Big).
\end{split}
\end{equation}
\end{theorem}

By the QUE theorem of Holowinsky and Soundararajan \eqref{hypothesis} holds for fixed, but arbitrarily small $h(k)$.
This reproduces the main result of Rudnick \cite{R}. Additionally, Theorem \ref{effective que} implies that 
\eqref{hypothesis} holds for $h(k) \gg (\log k)^{-\delta+\varepsilon}$
with $\tfrac {1}{7} \cdot (31/2-4\sqrt{15})=0.001152 \ldots$.
Assuming the Generalized Lindel\"of Hypothesis
it follows from an argument of Young \cite{Y} that
\eqref{hypothesis} holds for $h(k) \ge k^{-1/4 +\varepsilon}$. 
\footnote{In Proposition 5.1 of \cite{Y} Young establishes the analog of this for Hecke-Maass cusp forms.
The proof for holomorphic case follows in much the same way.} 
Adapting some ideas from recent unpublished work of Borichev and Sodin
to this setting it should be possible to give a better estimate for the second error term in \eqref{main eq}. In particular,
a consequence of this would improve the range of $r$ in Theorem \ref{zeros thm1} to $r \ge (\log k)^{-\delta}$, whereas we require $r \ge (\log k)^{-\delta/2}$.

\begin{proof}[Proof of Theorems \ref{zeros thm1} and \ref{zeros thmlin}]
Let $\phi_1$ be a smooth function such
that $\phi_1$ has compact support within $B(z_0,r)$,
$\phi_1(z)=1$ for $z \in B(z_0, r-M^{-1})$,
and $\Delta \phi_1 \ll M^{2}$, where $M$ tends to infinity with $k$
and will be chosen later.
Also suppose that $r \ge 2/M$.
 Similarly,
let $\phi_2$ be a smooth function such
that $\phi_2$ has compact support within $B(z_0, r+M^{-1},)$,
$\phi_2(z)=1$ for $z \in B(z_0, r)$,
and $\Delta \phi_2 \ll M^{2}$. We have that
\begin{equation} \notag
\begin{split}
\frac{k}{12} \cdot \frac{3}{\pi} \iint_{\mathcal F}
|\phi_1(z)-\phi_2(z)| \frac{dx dy}{y^2}
\ll& k \cdot \tmop{Area}_{\mathbb H}(B(z_0, r+M^{-1}) \setminus  B(z_0, r-M^{-1})) \\
\ll & k \cdot r \cdot   M^{-1}.
\end{split}
\end{equation}
Also $ \iint_{\mathcal F}
|\Delta \phi_j(z) | \frac{dx \, dy}{y^2} \ll r M$. 
Next, observe that
\[
\sum_{\varrho_f } \phi_1(\varrho_f) \le \# \{ \varrho_f \in B(z_0,r) \} \le \sum_{\varrho_f } \phi_2(\varrho_f).
\]
Thus, Theorem \ref{equi prop} implies
\[
\# \{ \varrho_f \in B(z_0,r) \}=\frac{k}{12} \frac{\tmop{Area}_{\mathbb H}(B(z_0,r))}{\tmop{Area}_{\mathbb H}(\mathcal F)}+O(rM \cdot  k  \cdot h(k) \log 1/h(k))+O(k \cdot r \cdot   M^{-1}).
\]
We take $M=h(k)^{-1/2}$, $h(k)=(\log k)^{-\delta + \varepsilon}$, in the unconditional case, and $M = h(k)^{-1/2}, h(k) = k^{-1/4 + \varepsilon}$ in the conditional case. Using Theorem \ref{effective que} then completes the proof. The exponent $\delta$ is the same exponent as in Theorem \ref{effective que}. 
\end{proof}
For the proof of Theorem \ref{zeros thm2} we recall the work of Luo and Sarnak \cite{LS1}. Define
the probability measure $\nu := (3/\pi) dx dy / y^2$ and denote by $\mathcal{H}_k$ the space of Hecke cusp forms for the full modular group $\tmop{SL}_2(\mathbb{Z})$. Then, Luo and Sarnak (see Corollary 1.2 in \cite{LS1}) showed that
\begin{equation} \label{sphere}
\frac{1}{\# \mathcal{H}_k} \sum_{f \in \mathcal{H}_k} \sup_{B} |\mu_f(B) - \nu(B)|^2 \ll k^{-1/21}
\end{equation}
where the supremum is taken over all geodesic balls $B \subset \mathcal F$.

\begin{proof}[Proof of Theorem \ref{zeros thm2}]
For $r_1 \ge k^{-1/2}$ let
\[
\mathcal E_k(r_1) := \{  f \in \mathcal{H}_k : \exists z_0 \mbox{ s.t. }
 B(z_0,r_1) \subset \mathcal F \mbox{ and } \forall z \in  B(z_0, r_1), \, \,  y^k |f(z)|^2 \le k^{-2}  \}.
\]
Notice that if $f \in \mathcal{H}_k \setminus \mathcal E_k(r_1)$ then we may apply Theorem \ref{equi prop}
with $h(k) \ll r_1$
and argue as in the previous proof to get that for $r \ge \sqrt{r_1}$
\[
\# \{ \varrho_f \in B(z_0,r) \}=\frac{k}{12} \cdot \frac{\tmop{Area}_{\mathbb H}(B(z_0,r))}{\tmop{Area}_{\mathbb H}(\mathcal F)}+O(r  k  \cdot \sqrt{r_1} \log 1/r_1).
\]
It remains to bound the size of $\mathcal E_k(r_1)$.  We apply \eqref{sphere} to see that
\begin{equation} \notag
\begin{split}
 r_1^4 \cdot \# \mathcal E_k(r_1) \ll& \sum_{f \in \mathcal E(r)} \tmop{sup}_{z_0 \in \mathcal F}|\mu_f(B(z_0, r_1))-\nu(B(z_0, r_1))|^2 \\
\ll& \sum_{f \in \mathcal{H}_k} \sup |\mu_f(B)-\nu(B)|^2
\ll k^{20/21},
\end{split}
\end{equation}
where supremum in the second line is over all  hyperbolic balls, $B \subset \mathcal F$. The claim follows taking $r_1 = k^{-\delta}$.
\end{proof}

\subsection{Proof of Theorem \ref{equi prop}}
Let $\phi$ be a smooth function that is compactly supported on $\mathcal F$.
Our starting point is the following formula of Rudnick (see Lemma 2.1 of \cite{R},
note that we assume $\phi$ is supported in $\mathcal F$)
\begin{equation} \label{zeev}
\sum_{\varrho_f} \phi(\varrho_f)
=\frac{k}{12} \cdot \frac{3}{\pi} \iint_{\mathcal F} \phi(z) \frac{dx dy}{y^2}
+\frac{1}{2\pi} \iint_{\mathcal F} \log (y^{k/2} |f(z)|) \Delta \phi(z) \frac{dx dy}{y^2}.
\end{equation}
To prove Theorem \ref{equi prop} we need
to bound the second term in the above formula. The difficulty
here comes in estimating the contribution to the integral over the set where $f$
is exceptionally small.

We first require two auxiliary lemmas, the first of which is due to Cartan.
\begin{lemma}[Theorem 9 of \cite{levin}]
Given any number $H>0$ and complex numbers $a_1,a_2, \ldots, a_n$,
there is a system of circles in the complex plane, with the sum of the 
radii equal to $2H$, such that for each point $z$ lying outside these 
circles one has the inequality
\[
|z-a_1| \cdot |z-a_2| \cdots |z-a_n| > \Big( \frac{H}{e}\Big)^n.
\]
\end{lemma}

For $z_0 \neq \varrho_f$ define
\[
M_r(z_0):= \max_{|z-z_0| \le r} \bigg|
\frac{f(z)}{f(z_0)} \bigg|+3.
\]
The next lemma is from Titchmarsh \cite{T} (see Lemma $\alpha$ of section 3.9, especially formula (3.9.1)).
\begin{lemma} \label{lem:landau}
Let $g(z)$ be a holomorphic function on $|z-z_0| \le r$, with $g(z_0) \neq 0$.
Then there is an absolute constant $A>1$ such that for $|z-z_0| \le r/4$
\[
\bigg| \log \left|\frac{g(z)}{g(z_0)} \right|-\sum_{|\rho-z_0| \le r/2} \log \left|\frac{z-\rho}{z_0-\rho}\right|\bigg|
< A \log M_r(z_0),
\]
where the summation runs over zeros $\rho$ of $g$.
\end{lemma}

Let $\mathcal D$ be the convex hull of $\tmop{supp} \phi$.
Let $\eta,\varepsilon>0$.
We cover $\mathcal D$ with $N$ disks of radius $\varepsilon$
centered at the points $a_1, \ldots, a_N$.
The disks are chosen so that $N \ll \tmop{Area}(\mathcal D)/\varepsilon^2$. 
Define
\[
\mathcal T_{\delta} = \{ z \in \mathcal F :
|f(z) y^{k/2}| < e^{-\delta k} \}
\qquad \mbox{and} \qquad
\mathcal T_{\delta,j}= \mathcal T_{\delta}\bigcap D_{\varepsilon}(a_j) .
\]
Let $n_j=\# \{\varrho_f : \varrho_f \in D_{16\varepsilon}(a_j) \}$
and set
\[
 \mathcal S_{\eta, j}=\Big\{z \in D_{\varepsilon}(a_j) : \prod_{\varrho_f \in D_{8\varepsilon}(a_j)} |z-\varrho_f| < \Big( \frac{\eta \varepsilon}{e}\Big)^{n_j} \Big\}.
\]
By Cartan's lemma the area of $\mathcal{S}_{\eta,j}$
is $\le 4 \pi \eta^2 \varepsilon^2$.

\begin{lemma} \label{max}
Suppose that
  $\varepsilon  >  \log k/k$ and $f(z_0) y_0^{k/2} \gg e^{-\varepsilon \cdot k}$.
Then there exists a constant $C>1$
such that
\[
M_{16 \varepsilon}(z_0) \ll e^{C \varepsilon \cdot k}.
\]
\end{lemma}
\begin{proof}
There is a point $z_{\tmop{max}}=x_{\tmop{max}}+iy_{\tmop{max}}$ such that
\[
\max_{z \in D_{16 \varepsilon}(z_0) } \bigg| \frac{f(z)}{f(z_0)} \bigg|
=\bigg| \frac{f(z_{\tmop{max}})}{f(z_0)} \bigg|=\Big(\frac{y_{0}}{y_{\tmop{max}}}\Big)^{k/2} \cdot 
\bigg| \frac{y_{\tmop{max}}^{k/2} f(z_{\tmop{max}})}{y_0^{k/2} f(z_0)} \bigg|.
\]
By Proposition A.1 of Rudnick \cite{R} we have $|y_{\tmop{max}}^{k/2} f(z_{\tmop{max}})| \ll k^{1/2}$.
(Note that Xia \cite{Xia} has recently improved this bound to $\ll k^{-1/4+\epsilon}$,
but we do not need that here.) Also, $y_0^{k/2} f(z_0) \gg e^{-\varepsilon k}$
and
\[
\Big(\frac{y_{0}}{y_{\tmop{max}}}\Big)^{k/2}\le \Big(\frac{y_{0}}{y_{0}-16\varepsilon} \Big)^{k/2}
\le e^{C \varepsilon k }.
\]
Combining these bounds we see that
\[
M_{16 \varepsilon}(z_0) \ll k^{1/2} e^{\varepsilon \cdot k} \cdot e^{C \varepsilon k} \ll e^{ C' \varepsilon k}.
\]
\end{proof}

\begin{lemma} \label{exceptional}
Suppose $\varepsilon>\log k/k$
and that for all $z_0 \in \mathcal F$ there exists a point $z_1=x_1+iy_1 \in D_{\varepsilon}(z_0)$ such that
$y_1^k|f(z_1)|^2  \gg e^{- \varepsilon k}$. Then there is
an absolute constant $\tfrac12>c_0>0$ such that for $\delta \ge 1/c_0 \cdot \varepsilon$ we have whenever
$\eta > \exp(-c_0 \delta/\varepsilon)$ that
\[
\mathcal T_{\delta,j} \subset \mathcal S_{\eta,j}
\]
for each $j=1, \ldots, N$.
\end{lemma}

\begin{proof}
By assumption, for each $j=1, \ldots, N$
there exists a point $z_j \in D_{\varepsilon}(a_j)$ such that
$|f(z_j)| \gg e^{-\varepsilon k} y_j^{-k/2}$.
If $z \in \mathcal T_{\delta,j}$
then 
\begin{equation} \label{eq:lowerbound}
\bigg|\frac{ f(z)}{f(z_j)} \bigg| \ll \Big(\frac{y_j}{y} \Big)^{k/2}  e^{-\delta k+\varepsilon k} \le \Big(\frac{y+2\varepsilon}{y} \Big)^{k/2}  e^{-\delta k+\varepsilon k} \le e^{-\delta k+3\varepsilon k}\le e^{-\delta k /4} .
\end{equation}
By Lemma \ref{lem:landau}
if $z_0 \neq \varrho_f$ there is a constant $A>1$ such that
 for $|z-z_0| \le \tfrac14 r$
\begin{equation} \notag
\bigg|\log \bigg|\frac{f(z)}{f(z_0)}\bigg| +\sum_{\varrho_f \in D_{r/2}(z_0)} \log \bigg|\frac{ z_0-\varrho_f}{z-\varrho_f}
\bigg| \bigg| < A \cdot \log M_r(z_0).
\end{equation}
Using this with $z_0=z_j$ along with \eqref{eq:lowerbound} we get that
for $z \in \mathcal T_{\delta, j}$
\begin{equation} \label{one bd}
-A \log M_{8 \varepsilon}(z_j) < -\delta k/5 + \sum_{\varrho_f\in D_{4 \varepsilon}(z_j)} \log \bigg|\frac{z_j-\varrho_f}{z-\varrho_f} \bigg|.
\end{equation}

For $z \in D_{\varepsilon}(a_j) \setminus \mathcal S_{\eta,j}$
\begin{equation} \label{two bd}
\sum_{\varrho_f \in D_{4\varepsilon}(z_j)} \log \bigg|\frac{z_j-\varrho_f}{z-\varrho_f} \bigg|
\le  \log \prod_{\varrho_f \in D_{4\varepsilon}(z_j)}\frac{4\varepsilon}{|z-\varrho_f|} \le
n_j \log \frac{4e}{\eta }< A' \log M_{16 \varepsilon}(z_j ) \log \frac{4e}{\eta} ,
\end{equation}
for some absolute constant $A'>0$ and the last inequality follows from Jensen's formula
(we have also used the inequality $\prod_{\varrho_f \in D_{4\varepsilon}(z_j)} |z-\varrho_f|>\prod_{\varrho_f \in D_{8\varepsilon}(a_j)} |z-\varrho_f|$ for $|z-a_j|<\varepsilon$). 

For the sake of contradiction, suppose that $\mathcal T_{\delta,j}$ is not
contained in $\mathcal S_{\eta,j}$. Then
combining \eqref{one bd} and \eqref{two bd} it follows that
\[
\log M_{16 \varepsilon}(z_j)>  \frac{\delta k}{5(A+A'\log 4e/\eta)}.
\]
However, by Lemma \ref{max} 
$\log M_{16\varepsilon}(z_j)\ll \varepsilon k $, so that a contradiction
is reached when $c_0$ is sufficiently small.
\end{proof}

A simple consequence of the previous lemma gives us a
bound on the size of our exceptional set $\mathcal T_{\delta}$.
This is one of the main ingredients in the proof of Theorem \ref{equi prop}.
Observe that under the hypotheses of the previous lemma
\begin{equation} \label{meas small}
\tmop{meas}(\mathcal T_{\delta} \cap \mathcal D)  \le \sum_{j=1}^N \tmop{meas}(\mathcal T_{\delta,j}) \le
\sum_{j=1}^N \tmop{meas}(\mathcal S_{\eta,j})  \le N 4 \pi^2 \eta^2 \varepsilon^2
\ll \eta^2.
\end{equation}

We also require the following crude, yet sufficient bound on the second moment
of $\log y^{k/2} |f(z)|$.
\begin{lemma} \label{sq bd}
We have
\[
\iint_{\mathcal D} (\log (y^{k/2} |f(z)|))^2  dx \, dy \ll k^2.
\]
\end{lemma}
\begin{proof}

Let $c_0$ be as in Lemma~\ref{exceptional}. We take $\varepsilon$ fixed but small, $\delta=1/c_0 \cdot \varepsilon$ and $\eta \in (\exp(-c_0 \delta/\varepsilon), 1/2)$.
For each $j=1,2,\ldots,N$ (note that here $N =O( 1)$) there exists $c_j \in D_{\varepsilon}(a_j)$ 
such that $c_j \notin \mathcal S_{\eta,j}$, which by Lemma \ref{exceptional}
implies that $c_j \notin \mathcal T_{1/c_0 \cdot\varepsilon, j}$.  Thus,
$f(c_j) \gg e^{-1/c_0 \cdot \varepsilon k} (\tmop{Im}(c_j))^{-k/2}$ and $\prod_{\varrho_f \in D_{8 \varepsilon}(c_j)} |c_j-\varrho_f|
\ge (\varepsilon \eta/e)^{n_j}$.

Now apply Lemma \ref{lem:landau}
to see that for $|z-c_j| \le 2 \varepsilon$
\[
\log \bigg| \frac{f(z)}{f(c_j)}\bigg|
=\sum_{\varrho_f \in D_{4 \varepsilon}(c_j)}
\log\bigg| \frac{z-\varrho_f}{c_j-\varrho_f} \bigg|
+O(\log M_{8 \varepsilon}(c_j)).
\]
Apply Lemma \ref{max} and our earlier observations to see that
for $|z-c_j| \le 2\varepsilon$ we have
\[
\log |f(z)|=\sum_{\varrho_f \in D_{4 \varepsilon}(c_j)}
\log\big| z-\varrho_f\big|
+O(k).
\]
This implies that
\[
\int_{|z-a_j| \le  \varepsilon}
(\log |f(z)|)^2 dz \ll n_j \sum_{\varrho_f  \in D_{4 \varepsilon}(c_j)} \int_{|z-c_j| \le  2\varepsilon} (\log |z-\varrho_f|)^2 dz+ k^2 \ll k^2.
\]
Summing over all the disks we see that
\[
\iint_{\mathcal D} (\log (y^{k/2} |f(z)|))^2  \frac{dx \, dy}{y^2} \ll  k^2 \int_{\mathcal D} (\log y)^2  dz
+\int_{\mathcal D} (\log |f(z)|)^2 dz \ll k^2.
\]
\end{proof}

We are now prepared to prove Theorem \ref{equi prop}.
\begin{proof}[Proof of Theorem \ref{equi prop}]
By \eqref{zeev} it suffices to show that
\[
\bigg| \frac{1}{2\pi} \iint_{\mathcal F} \log(y^{k/2} |f(z)| )\Delta \phi(z) \frac{dx \, dy}{y^2}
\bigg| \ll k \cdot h(k) \log 1/h(k) \cdot
\iint_{\mathcal F } |\Delta \phi(z)| \frac{dx \, dy}{y^2}+k \cdot h(k)^2.
\]
Note that for $\delta> \log k/k$
\[
\bigg| \iint_{\mathcal F \setminus \mathcal T_{\delta}} \log(y^{k/2} |f(z)| ) \Delta \phi(z) \frac{dx \, dy}{y^2}
\bigg| \ll k\delta \iint_{\mathcal F } |\Delta \phi(z)| \frac{dx \, dy}{y^2}.
\]
Next, note that
\[
\begin{split}
\bigg|
\iint_{\mathcal T_{\delta}} \log(y^{k/2} |f(z)| )\Delta \phi(z) \frac{dx \, dy}{y^2} \bigg|
\le& \bigg(\iint_{\mathcal T_{\delta}} |\Delta \phi(z)|^2 \frac{dx \, dy}{y^2} \bigg)^{1/2} \cdot 
\\
& \qquad \qquad \times \bigg(
\iint_{\mathcal F} (\log(y^{k/2} |f(z)| ) )^2 \frac{dx \, dy}{y^2} \bigg)^{1/2}.
\end{split}
\]
Recall our assumption \eqref{hypothesis}, which states that for every $z_0 \in \mathcal F$ there exists a point $z_1 \in D_{h(k)}(z_0)$ with $y_1^k|f(z_1)| \gg e^{-kh(k)}$. Hence, formula \eqref{meas small} implies that, for $\varepsilon \ge h(k)$ and $\eta > \exp(-c_0 \delta/\varepsilon)$,
\[
\iint_{\mathcal T_{\delta}} |\Delta \phi(z)|^2 \frac{dx \, dy}{y^2} \ll h(k)^{-2A} \iint_{\mathcal T_{\delta} \cap \mathcal D} 1 \,  \frac{dx \, dy}{y^2}
\ll \eta^2 h(k)^{-2A}.
\]
Therefore, collecting estimates and applying Lemma \ref{sq bd} we have
\[
\frac{1}{2\pi} \iint_{\mathcal F} \log(y^{k/2} |f(z)| )\Delta \phi(z) \frac{dx \, dy}{y^2}
\ll k \delta \iint_{\mathcal F } |\Delta \phi(z)| \frac{dx \, dy}{y^2}+k \eta  \cdot h(k)^{-A}.
\]
We now take $\varepsilon= h(k)$, $\delta=((A+2) /c_0) \cdot \varepsilon \log 1/\varepsilon$
and $\eta= 2\exp(-c_0 \delta/\varepsilon)$.
\end{proof}

\section{Zeros of cusp forms high in the cusp}

To detect zeros of $f$ high in the cusp we use the following special case of a result of Ghosh and Sarnak \cite[Theorem 3.1]{GS} that shows that
for certain values of $\tmop{Im}(z)$ the Hecke cusp form
$f(z)$ is essentially determined by one term in its Fourier expansion. In this section
 we normalize $f$ so that the first term in its Fourier expansion equals one.
\begin{lemma}[Proposition 2.1 of \cite{Ma}] \label{matomaki}
There are positive constants $c_2,c_3$ and $\delta$ such that, for all
integers $\ell \in (c_2, c_3 \sqrt{k /\log k})$ and $f \in \mathcal{H}_k$
\[
\Big( \frac{e}{\ell} \Big)^{\frac{k-1}{2}} f(x+iy_{\ell})=\lambda_f(\ell)e(x\ell)+O(k^{-\delta}),
\]
where $y_{\ell}=\frac{k-1}{4\pi \ell}$.
\end{lemma}

This essentially tells us that 
on the vertical geodesic $\tmop{Re}(z)=0$ a sign change of $\lambda_f(\ell)$
yields a zero of $f$. More precisely, to detect a zero on $\tmop{Re}(z)=0$ it suffices to find $\ell_1$ and $\ell_2$ in $(c_2, c_3 \sqrt{k /\log k})$
such that
\[
\lambda_{f}(\ell_1) < -k^{-\epsilon} < k^{-\epsilon} <\lambda_f(\ell_2)
\]
where $\epsilon>\delta$. A similar analysis holds on the geodesic $\tmop{Re}(z)=-1/2$, but here
one also needs
 $\ell_1$ and $\ell_2$ to be odd.

\subsection{Proof of Theorem~\ref{thm:realzeravg}}
We detect sign changes for almost all forms using a very recent theorem of the last two authors \cite[Theorem 1 with $\delta = (\log h)^{-1/200}$]{MaRa}.
\begin{lemma}
\label{le:KaisaMaks}
Let $g : \mathbb{N} \rightarrow [-1,1]$ be a multiplicative function. There exists an absolute constant $C>1$ such that, for any $2 \leq h \leq X$, 
$$
\Bigg | \frac{1}{h} \sum_{x \leq n \leq x + h} g(n) - \frac{1}{X} \sum_{X \leq n \leq 2X} g(n) \Bigg | \leq (\log h)^{-1/200}
$$
for almost all $X \leq x \leq 2X$ with at most $C X (\log h)^{-1/100}$ exceptions.
\end{lemma}

To benefit from this, we need to control the number of $n$ for which $|\lambda_f(n)| < n^{-\delta}$ and the number of $p$ for which $\lambda_f(p) < 0$. For this we quote two lemmas. The first one is an immediate consequence of~\cite[Theorem 2]{MS09}.

\begin{lemma}
\label{le:MuSi}
Let $p$ be a prime. Then
\[
\frac{\#\{f \in \mathcal{H}_k \colon |\lambda_f(p)| < p^{-\delta}\}}{\# \mathcal{H}_k} \ll p^{-\delta} + \frac{\log p}{\log k},
\]
where the implied constant is absolute and effectively computable.
\end{lemma}

The second lemma is a large sieve
inequality for the Fourier coefficients $\lambda_f(n)$. The version we use is the following special case of a more general theorem \cite[Theorem 1]{LauWu08} due to Lau and Wu.
\begin{lemma}
\label{le:larsiev}
Let $\nu \geq 1$ be a fixed integer. Then
\[  
\sum_{f \in \mathcal{H}_k} \left| \sum_{P < p \leq Q} \frac{\lambda_f(p^{\nu})}{p}\right|^{2} \ll_\nu k \frac{1}{P \log P} + k^{10/11} \frac{Q^{\nu/5}}{(\log P)^2}
\]  
uniformly for
\[  
2 \mid k, \quad 2 \leq P < Q \leq 2P.
\]
\end{lemma}

\begin{proof}[Proof of Theorem~\ref{thm:realzeravg}]
Let $X = k/Y$ and $\delta > 0$. Define a multiplicative function $g \colon \mathbb{N} \to \{-1, 0, 1\}$ by
\[
g(p^\nu) = 
\begin{cases}
\tmop{sgn}(\lambda_f(p^\nu)) & \text{if $|\lambda_f(p^\nu)| \geq p^{-\delta \nu}$ and $p > 2$} \\
0 & \text{otherwise}.
\end{cases}
\]
Notice that if $g(n) \neq 0$, then $n$ is odd, $|\lambda_f(n)| \geq n^{-\delta}$ and $g(n) = \tmop{sgn}(\lambda_f(n))$. Hence by Lemma~\ref{matomaki}, a sign change of $g(n)$ yields a real zero of $f$ and thus the claim follows once we have shown that, for all but at most $\varepsilon \cdot \#\mathcal{H}_k$ of $f \in \mathcal{H}_k$, we have, for $h$ large but fixed and for proportion $9/10$ of $x \sim X$,
\[
\frac{1}{h} \sum_{\substack{x \leq n \leq x+h}} |g(n)| - \frac{1}{h} \left|\sum_{\substack{x \leq n \leq x+h}} g(n) \right| \gg 1.
\]
Lemma~\ref{le:KaisaMaks} applied to $g(n)$ and $|g(n)|$ reduces this to showing that, for all but at most $\varepsilon \cdot \#\mathcal{H}_k$ of $f \in \mathcal{H}_k$, we have
\begin{equation}
\label{eq:sgndiff}
\frac{1}{X} \sum_{n \sim X} |g(n)| - \frac{1}{X} \left|\sum_{\substack{n \sim X}} g(n) \right| \gg 1.
\end{equation}

By Lemma~\ref{le:MuSi}
\[
\sum_{f \in \mathcal{H}_k} \sum_{\substack{p \leq X \\ |\lambda_f(p)| < p^{-\delta}}} \frac{1}{p} \ll \# \mathcal{H}_k \cdot \sum_{p \leq X} \left(p^{-1-\delta} + \frac{\log p}{p\log k}\right) = O(\# \mathcal{H}_k).
\]
Hence there is an absolute positive constant $C$ such that for given any $\varepsilon > 0$, 
\begin{equation}
\label{eq:nottooclose0}
\sum_{\substack{p \leq X \\ g(p) = 0}} \frac{1}{p} \leq \frac{C}{\varepsilon}
\end{equation}
for all but at most $\varepsilon/2 \cdot \#\mathcal{H}_k$ forms $f \in \mathcal{H}_k$. Consequently, with this many exceptions,
\begin{equation}
\label{eq:nonsmalllow}
\frac{1}{X} \sum_{\substack{n \sim X}} |g(n)| = \frac{1}{X} \sum_{\substack{n \sim X \\ g(n) \neq 0}} 1 \gg 1.
\end{equation}

On the other hand, since $|\lambda_f(p)| \leq 2$ for all primes $p$, for any $Q \geq P \geq 2$,
\[
\begin{split}
\sum_{\substack{P \leq p \leq Q \\ \lambda_f(p) < 0}} \frac{1}{p} 
&\geq \sum_{P \leq p \leq Q} \frac{(\lambda_f(p)^2 - 2\lambda_f(p))}{8p} = \frac{1}{8} \sum_{P \leq p \leq Q} \frac{\lambda_f(p^2) - 2 \lambda_f(p) + 1}{p},
\end{split}
\]
so that
\[
\begin{split}
\sum_{\substack{p \leq X \\ \lambda_f(p) < 0}} \frac{1}{p} &\geq  \sum_{\substack{\log X \leq p \leq X^{1/1000} \\ \lambda_f(p) < 0}} \frac{1}{p} \geq \frac{1}{8} \sum_{\log X \leq p \leq X^{1/1000}} \frac{\lambda_f(p^2) - 2 \lambda_f(p) + 1}{p} \\
& =\frac{1+o(1)}{8} \log \log X + \sum_{\log X \leq p \leq X^{1/1000}} \frac{\lambda_f(p^2) - 2 \lambda_f(p)}{p}.
\end{split}
\]
Splitting the last sum into dyadic intervals and then applying Lemma~\ref{le:larsiev} we see that it contributes $o(\log \log X)$ for almost all forms $f$. Hence, recalling~\eqref{eq:nottooclose0} and the definition of $g(n)$,
\[
\sum_{\substack{p \leq X \\ g(p) = -1}} \frac{1}{p} \geq \frac{1+o(1)}{8} \log \log X
\]
for all but $\varepsilon/2 \# \mathcal{H}_k$ forms $f \in \mathcal{H}_k$. By Halasz's theorem for real valued functions (see for instance~\cite{HT91}), this implies
\[
\frac{1}{X}\sum_{\substack{n \leq X}} g(n) = o(1).
\]
Hence~\eqref{eq:sgndiff} follows from this and \eqref{eq:nonsmalllow}, which completes the proof.
\end{proof}

\subsection{Proof of Theorem~\ref{thm:signs}}
Our main proposition for the proof of Theorem~\ref{thm:signs} shows that the Lindel\"of hypothesis implies many sign changes of $\lambda_f(\ell)$. For the remainder of this section we use the notation $x \sim X$ to mean $X \le x \le 2X$.
\begin{proposition}
\label{prop:signchanges}
Assume the Generalized Lindel\"of Hypothesis, let $\varepsilon, \eta > 0$ and $X \geq k^\eta$. Then, for almost all $x \sim X$, the interval $[x, x+X^\varepsilon]$ contains integers $m_1$ and $m_2$ such that $\lambda_f(m_1) < - X^{-\varepsilon}$ and  $\lambda_f(m_2) > X^{-\varepsilon}$. 
%
\end{proposition}

Observe that the lower bound \eqref{first part}, the first part of
Theorem \ref{thm:signs}, immediately follows from this for $X = k/Y$ along with Lemma \ref{matomaki}. We will delay the proof of \eqref{second part} until the end of 
the section.

To prove the proposition, we study first and second moments of $\lambda_f(n)$ in short intervals. 

\begin{lemma} \label{short 1}
Assume the Generalized Lindel\"of Hypothesis. Let $\epsilon, \eta > 0$, $X\ge k^{\eta}$, and  $2 \le L \le X$. Then
\[
\bigg|
\sum_{x < n \leq x +  \frac{x}{L}} \lambda_f(n) \bigg| \ll X^{\epsilon} \Big(\frac{X}{L}\Big)^{1/2}
\]
for all $x \sim X$ with at most $X^{1-\epsilon}$ exceptions.
\end{lemma}
\begin{proof}
This follows once we have shown that for any $\varepsilon>0$
\begin{equation}
\label{eq:meansquarelambf2}
\frac{1}{X}\int_X^{2X} \bigg|\sum_{x < n \leq x + \frac{x}{L}} \lambda_f(n)\bigg|^2 dx  \ll k^{\varepsilon} \frac{X}{L^{1-\varepsilon}}.
\end{equation}
We follow an argument of Selberg \cite{Sel} on primes short intervals. 
Let $\delta=\log (1+\frac{1}{L}) \approx \frac{1}{L}$.
We get by Perron's formula that for $x, x+\frac{x}{L} \notin \mathbb Z$
\begin{equation} \notag
\begin{split}
\sum_{x < n \leq x+\frac{x}{L}} \lambda_f(n) =& \frac{1}{2 \pi i} \int_{1/2-i\infty}^{1/2+i\infty} L(s, f)  \frac{(x+\frac{x}{L})^s-x^s}{s} ds\\
=&x^{1/2} \cdot \frac{1}{2 \pi } \int_{-\infty}^{\infty} L(\tfrac12+it, f) \, w_{\delta}(\tfrac12+it) \cdot e^{it \cdot \log x} dt,
\end{split}
\end{equation}
where $w_{\delta}(s)=(e^{s\delta}-1)/s$.

Observe that $|w_{\delta}(s)| \ll \min(\delta/2 , 1/|s|)$. Thus,
making a change of variables and applying Plancherel we see that
\[
\begin{split}
&\frac{1}{X^2} \int_{X}^{2X} \left|\sum_{x < n \leq x + \frac{x}{L} } \lambda_f(n)\right|^2 dx \le \int_{0}^{\infty} \left|\sum_{x < n \leq x + \frac{x}{L} } \lambda_f(n)\right|^2 \frac{dx}{x^2} \\
&= \int_{-\infty}^{\infty} \left|\sum_{e^{\tau} < n \leq e^{\tau+\delta} } \lambda_f(n)\right|^2 \frac{d\tau}{e^{\tau}} = \frac{1}{4\pi^2} \int_{-\infty}^{\infty} |L(\tfrac12+it, f)|^2  |w_{\delta}(\tfrac12+it)|^2 \, dt \\
&\ll k^{\varepsilon} \bigg(\int_{-1/\delta}^{1/\delta}  \delta^2 |t|^{\varepsilon} \, dt
+\int_{|t|>1/\delta} \frac{1}{|t|^{2-\varepsilon}} \, dt\bigg)
\ll k^{\varepsilon} \delta^{1-\varepsilon} \ll \frac{k^{\varepsilon}}{L^{1-\varepsilon}}.
\end{split}
\]
This establishes \eqref{eq:meansquarelambf2} and the claim follows.
\end{proof}

\begin{lemma} \label{short 2}
Assume the Generalized Lindel\"of Hypothesis.  Let $\epsilon,\eta>0$,  $X \ge k^{\eta}$
and $2 \le L \le X$. Then
\[
\sum_{x < n \leq x + \frac{x}{L}} |\lambda_f(n)|^2 =\frac{6}{\pi^2} L(1, \tmop{sym}^2 f )  \cdot \frac{x}{L}
+O\Big(X^{\epsilon} \Big(\frac{X}{L}\Big)^{1/2}\Big)
\]
for all $x \sim X$ with at most $X^{1-\epsilon}$ exceptions.
\end{lemma}

\begin{proof}
One has
\[
\sum_{n \ge 1} \frac{\lambda_f(n)^2}{n^s} = \zeta(2s)^{-1} L(s,f \otimes f) = \frac{\zeta(s)}{\zeta(2s)} L(s,\tmop{sym}^2 f).
\] 
Writing $w_{\delta}(s)=(e^{s\delta}-1)/s$
and
arguing as in the proof of Lemma~\ref{short 1} before only now noting the pole at $s = 1$
we have that
\[
\begin{split}
&\frac{1}{X^2}\int_{X}^{2X} \left|\sum_{x < n \leq x +\frac{x}{L}} \lambda_f(n)^2 - \frac{x}{L}\tmop{Res}_{s=1} \frac{\zeta(s)L(s, \tmop{sym}^2 f)}{\zeta(2s)} \right|^2 dx \\
&\ll \int_{-\infty}^{\infty} \left|\frac{\zeta(\tfrac12+it)}{\zeta(1+2it)} L(\tfrac12, \tmop{sym}^2 f)\right|^2
|w_{\delta}(\tfrac12+it)|^2 dt 
\ll \frac{k^{\varepsilon}}{L^{1-\varepsilon}},
\end{split}
\]
and the claim follows.
\end{proof}

\begin{proof}[Proof of Proposition~\ref{prop:signchanges}]
It is shown in \cite{HL} that for any $\epsilon>0$
\[
L(1, \tmop{sym}^2 f) \gg k^{-\epsilon}.
\]
Additionally, we have Deligne's bound $\lambda_f(n) \ll n^{\epsilon}$.
Hence, by these facts along with Lemmas \ref{short 1} and \ref{short 2}, we have for almost all $x \sim X$ that
\[
\sum_{\substack{x \leq n \leq x+X^{\varepsilon} \\ \lambda_f(n) \gtrless 0}} \lambda_f(n) \gg X^{9\varepsilon/10}.
\]
The claim follows since $n$ with $|\lambda_f(n)| \leq X^{-\varepsilon}$ contribute at most $O(1)$ to the sum.
\end{proof}

\begin{proposition}
\label{prop:signchanges1/2line}
Assume the Generalized Lindel\"of Hypothesis. Let $\varepsilon, \eta > 0$ and $X \geq k^\eta$. Then, for almost all $x \sim X$, the interval $[x, x+X^\varepsilon]$ contains odd integers $m_1$ and $m_2$ such that $\lambda_f(m_1) < -X^{-\varepsilon}$ and $\lambda_f(m_2) > X^{-\varepsilon}$. 
\end{proposition}
\begin{proof}
The proof goes similarly to the proof of Proposition~\ref{prop:signchanges}. Here we
have the extra condition $(n, 2) = 1$ in the sums.  To account for this condition first note that, for $\tmop{Re}(s)>1$, $L(s,f)$
and $L(s,\tmop{sym}^2f)$ have Euler product representations given in terms of
a product of local factors at each prime. That is,
\[
L(s,f)=\prod_p L_p(s,f) \qquad \mbox{and} \qquad L(s, \tmop{sym}^2 f) =\prod_p L_p(s,\tmop{sym}^2f).
\]

The argument goes along the same lines as before, except in place of $L(s,f)$ and $L(s,\tmop{sym}^2 f)$ one uses
\[
L(s,f) \cdot (L_2(s, f))^{-1} \qquad \mbox{and} \qquad L(s,f) \cdot (L_2(s,\tmop{sym}^2 f))^{-1}.
\]
The contribution from the local factor at $p=2$ is bounded.
\end{proof}

\section{Effective QUE}
For two smooth, bounded functions $h,g$ the Petersson inner product is given by
\[
\langle h,g \rangle=\iint_{\mathcal F} h(z) \overline{g(z)} \, \frac{dx \, dy}{y^2}.
\]
Let $F_k(z) = y^{k/2} f(z)$ and assume that $F_k$ is normalized so that 
$\lVert F_k \rVert^2:=\langle F_k, F_k\rangle=1$.
In this section we establish QUE with an unconditional, effective error term. 
Under the assumption of the Generalized Lindel\"of Hypothesis effective error
terms have been obtained by Watson \cite{W} and Young \cite{Y}.  
For the unconditional result our arguments essentially follow those of Holowinsky and Soundararajan~\cite{H, S, HS}, except for one modification which we have borrowed from Iwaniec's course notes on QUE. We have also used some ideas of Matt Young \cite{Y} and the final optimization uses a trick from Iwaniec's course notes on QUE.

As in Holowinsky's and Soundararajan's \cite{HS} proof of QUE, we shall estimate the inner product $\langle |F_k|^2, \phi \rangle$ of $|F_k|^2$ with a smooth function $\phi$ in two ways. In the first way, based on Soundararajan's approach~\cite{S}, we use the spectral decomposition of $\phi$ and bounds for inner products, the most involved case being the inner product of $|F_k|^2$ with Maass forms. Here a formula of Watson \cite{W} for $\langle |F_k|^2, u \rangle$, where $u$ is a Maass form, plays a crucial role. In the second way, based on Holowinsky's approach~\cite{H}, we compute the inner product $\langle |F_k|^2, \phi \rangle$ using a smoothed incomplete Eisenstein series and bounds for the shifted convolution problem. Each approach alone fails if the Fourier coefficients of $F_k$ misbehave in a certain way, but as noticed in~\cite{HS}, in the two approaches, the misbehavior is of different nature, and one of the approaches always works.

\subsection{Soundararajan's approach}
\label{ss:Sound}
The following treatment of the inner product of $|F_k|^2$ with a cusp form
is taken from Iwaniec's notes on QUE. 
\begin{lemma} \label{lem:Iwaniec}
Let $u_j$ be an $L^2$-normalized Hecke-Maass cusp form with spectral parameter $t_j$ with $|t_j| \le k$. Then,
$$
|\langle |F_k|^2, u_j \rangle | \ll |t_j|^{1/2+\varepsilon} (\log k)^{\varepsilon} \prod_{p \leq k} \Big ( 1 - \frac{n(p)}{p} \Big ) 
$$
where $n(p) = \lambda_f(p^2) + \tfrac 14 \cdot (1 - \lambda^2_f(p^2) )$. 
\end{lemma}

\begin{proof} 
By Watson's formula \cite{W}
$$
|\langle u_j F_k, F_k \rangle |^2 \ll  \frac{\Lambda(\tfrac 12, u_j \times f \times f)}{\Lambda(1, \text{sym}^2 u_j) \Lambda(1, \text{sym}^2 f)^2}  .
$$
The ratio of the Gamma factors is $\ll 1/k$, and therefore 
$$
| \langle u_j F_k , F_k \rangle | \ll \frac{|L(\tfrac 12, u_j \times \text{sym}^2 f)|^{1/2} \cdot |L(\tfrac 12, u_j)|^{1/2}}{\sqrt{k} |L(1, \text{sym}^2 f)| \cdot |L(1, \text{sym}^2 u_j)|^{1/2}}.
$$
For the $L$-functions depending only on $u_j$ we note that
the convexity bound gives $|L(\tfrac 12, u_j)| \ll t_j^{1/2+\varepsilon}$, 
while the work of Hoffstein and Lockhart \cite{HL} implies that $t_j^{-\varepsilon} \ll |L(1, \text{sym}^2 u_j)|$.
Next we note that
Lemma 2 of Holowinsky and Soundararajan \cite{HS} implies 
$$
|L(1, \text{sym}^2 f)|^{-1} \ll (\log\log k)^3 \prod_{p \leq k} \Big ( 1 - \frac{\lambda_f(p^2)}{p} \Big ).
$$ 
Therefore,
\begin{equation} \label{eq:watsonconvex}
|\langle u_j F_k, F_k \rangle| \ll (\log\log k)^3 \cdot \frac{t_j^{1/4 + \varepsilon}}{\sqrt{k}} \prod_{p \leq k} \Big ( 1 - \frac{\lambda_f(p^2)}{p}
\Big ) \cdot |L(\tfrac 12, u_j \times \text{sym}^2 f)|^{1/2}. 
\end{equation}

It suffices to bound the remaining $L$-function $L(\tfrac 12, u_j \times \text{sym}^2 f)$.
The analytic conductor $\mathfrak{C}$ of $L(\tfrac 12, u_j \times \text{sym}^2 f)$ satisfies $\mathfrak{C} \asymp (k + |t_j|)^4 \cdot |t_j|^2$.  
Therefore, by the approximate functional equation (see for instance Theorem 2.1 of Harcos \cite{Harcos}), and then Cauchy-Schwarz, 
\begin{equation} \label{eq:approxfun}
\begin{split}
|L(\tfrac 12, u_j \times \text{sym}^2 f)|^2 & \ll \Big ( \sum_{n \ge 1}
\frac{|\lambda_{u_j}(n) \lambda_f(n^2)|}{\sqrt{n}} \cdot \Big| V\Big( \frac{n}{\sqrt{\mathfrak C}}\Big) \Big| \Big )^2 \\
& \ll \sum_{n \ge 1} \frac{|\lambda_{u_j}(n)|^2}{\sqrt{n}}  \cdot \Big| V\Big( \frac{n}{\sqrt{\mathfrak C}}\Big)\Big| \times \sum_{n \ge 1}
\frac{\lambda_f(n^2)^2}{\sqrt{n}} \cdot \Big| V\Big( \frac{n}{\sqrt{\mathfrak C}}\Big)\Big| ,
\end{split}
\end{equation}
where $V$ is a smooth function satisfying $|V(x)| \ll_A \min(1, x^{-A})$ for any $A \ge 1$.
To bound the second term in \eqref{eq:approxfun} we 
use general bounds for multiplicative functions to see
\begin{equation} \label{eq:mult}
\sum_{n \leq \mathfrak{C}^{1/2} (\log \mathfrak C)^{\varepsilon} } \frac{\lambda_f(n^2)^2}{\sqrt{n}} \ll \mathfrak{C}^{1/4} (\log \mathfrak C)^{\varepsilon}
\prod_{p \leq \mathfrak{C}^{1/2} (\log \mathfrak C)^{\varepsilon}}
\Big ( 1 + \frac{\lambda_f(p^2)^2 - 1}{p} \Big ).
\end{equation}
Next we use Deligne's bound $|\lambda_f(n)| \le d(n)$, the elementary
estimate $\sum_{n \le X} d^2(n^2) \ll X(\log X)^8$, and partial summation to see that
for any $A \ge 1$ that
\begin{equation} \notag
\sum_{n \ge \mathfrak{C}^{1/2} (\log \mathfrak C)^{\varepsilon}} \frac{\lambda_f(n^2)^2}{\sqrt{n}} \cdot \Big| V\Big( \frac{n}{\sqrt{\mathfrak C}}\Big)\Big| \ll_A \frac{\mathfrak{C}^{1/4}}{(\log \mathfrak C)^A},
\end{equation}
which is bounded above by the right-hand side of \eqref{eq:mult}.

Next
observe that for $X \ge 2$
\begin{align}
 \sum_{n \ge 1} \frac{|\lambda_{u_j}(n)|^2}{\sqrt{n}} \cdot e^{-n/X}  \label{eq:shift}
 = \frac{1}{2\pi i} \int_{(2)} \frac{L(\tfrac 12 + s, u_j \otimes u_j)}{\zeta(2s+1)} \cdot \Gamma(s) X^{s} ds.
\end{align}
The convexity bound gives
$$
|L(\tfrac 12 + it, u_j \otimes u_j)| \ll |t_j|^{1/2+\varepsilon} \cdot (1 + |t|)^{1+\varepsilon}.
$$
By convexity we also have $|L(\sigma + it, u_j \otimes u_j)| \ll |t_j|^{1/2+\varepsilon} \cdot (1 + |t|)^{1+\varepsilon}$ uniformly in $\sigma \geq \tfrac 12$. 
In addition, from the works
Hoffstein and Lockhart \cite{HL} and Li \cite{XiLi} we have $|t_j|^{-\varepsilon} \ll L(1, \text{sym}^2 u_j) \ll |t_j|^{\varepsilon}$. Combining these ingredients it follows that (\ref{eq:shift}) equals
$$
\frac{6}{\pi^{3/2}} X^{1/2} L(1, \text{sym}^2 \ u_j) + O \Big( X^{\varepsilon} \cdot |t_j|^{1/2+\varepsilon} \Big).
$$
Using this and partial summation it follows
that the first term on the right-hand side of \eqref{eq:approxfun}
is $\ll \mathfrak{C}^{1/4} L(1, \text{sym}^2 u_j) +\mathfrak{C}^{\varepsilon} |t_j|^{1/2+\varepsilon}$.
Thus, applying this bound along with  \eqref{eq:mult} in \eqref{eq:approxfun} 
yields
\begin{equation} \notag
\begin{split}
& |L(\tfrac 12, u_j \times \text{sym}^2 f)|^{1/2} \\
&\ll  (\log \mathfrak{C})^{\varepsilon} \Big(\prod_{p \leq \mathfrak{C}^{1/2} (\log \mathfrak C)^{\varepsilon}} \Big ( 1 + \frac{\lambda_f(p^2)^2 - 1}{4 p} \Big ) \Big)  \Big ( \mathfrak{C}^{1/8} |t_j|^{\varepsilon} + \mathfrak{C}^{1/16+\varepsilon} |t_j|^{1/8+\varepsilon} \Big ).
\end{split}
\end{equation}
Using this in \eqref{eq:watsonconvex}, doing some minor manipulations in the Euler products,
and simplifying error terms we have that
$$
|\langle u_j F_k, F_k \rangle |\ll |t_j|^{1/2+\varepsilon} (\log k)^{\varepsilon} \prod_{p \leq k} \Big ( 1 - \frac{n(p)}{p} \Big ) 
$$
as claimed.

\end{proof}
We will also require a bound for the inner products of Maass cusp forms $u_j$ and Eisenstein series $E(\cdot, \tfrac 12 + it)$ with a smooth function $\phi$.  First we require a bound
for the Eisenstein series $E(z,s) =\sum_{\gamma \in \Gamma_{\infty} \backslash \Gamma} (\tmop{Im}(\gamma z))^s$ uniform in both $z$ and $s$.
The Eisenstein series has the Fourier expansion (see equation (3.29) of \cite{IwaniecSpectral}) 
\begin{equation} \label{eq:eisen}
E(z,s)=y^{s}+\frac{\theta(1-s)}{\theta(s)} y^{1-s}
+\frac{2 \sqrt{y}}{\theta(s)} \sum_{n \neq 0} \tau_{s-\frac12}(n) e(nx) K_{s-\frac12}(2\pi |n| y),
\end{equation}
where
$
\theta(s)=\pi^{-s} \Gamma(s) \zeta(2s)$ and
$\tau_{s-\frac12}(n)=\sum_{ab=|n|} (\frac{a}{b})^{s-\frac12}$. Using the following uniform estimates for the $K$-Bessel function
due to Balogh \cite{Balogh} (see Corollary 3.2 of Ghosh, Reznikov, and Sarnak \cite{GRS13})
\begin{equation} \label{eq:kbessel}
K_{it}(u) \ll \min\left( (t^2-u^2)^{-1/4} e^{-\frac{\pi}{2}t}, u^{-1/2} e^{-u}, t^{-1/3} e^{-\frac{\pi}{2}t} \right)
\end{equation}
along with Stirling's formula and the bound $|\zeta(1+it)|^{-1} \ll \log (|t|+1) $ one has
\begin{equation} \label{eq:eisenbd}
E(z,\tfrac12+it) \ll \sqrt{y}(1+|t|).
\end{equation}
For a more complete argument and a better bound see Lemma 2.1 of Young \cite{Y2}.

\begin{lemma} \label{lem:IBP}
Let $\phi$ be a smooth function
with support contained within $\{z  : 1/2 \le \tmop{Im}(z) \le C\}$
with $C>1$.
Also, suppose $\phi$ satisfies $\Delta^{\ell} \phi \ll (CM)^{2\ell}$ for all $\ell \geq 1$.  Then
$$
|\langle u_j , \phi \rangle| \ll_{A} \frac{(CM)^{2A}}{1 + |t_j|^{2A}}  \quad \text{ and } \quad
|\langle E(\cdot, \tfrac 12 + it), \phi \rangle| \ll_{A}  \frac{(CM)^{2A}}{1 + |t|^{2A-1}}
$$
for all $A \geq 1$.
\end{lemma}
\begin{proof}
The hyperbolic Laplacian is symmetric with respect to the Petersson inner product, that is, $\langle \Delta g, h \rangle = \langle g, \Delta h \rangle$. Therefore since $u_j$ is an eigenfunction of $\Delta$ with eigenvalue $\tfrac 14 + t_j^2$, we get
$$
 (\tfrac 14 + t_j^2)^{\ell} \langle u_j, \phi \rangle = 
\langle \Delta^{\ell} u_j, \phi \rangle = \langle u_j, \Delta^{\ell} \phi \rangle
\ll \langle |u_j|, 1  \rangle \cdot (CM)^{2\ell} .
$$
Since $\mathcal F$ has finite hyperbolic area we can bound the $L^1$-norm of $u_j$ by its $L^2$-norm, which is one. This gives the first claim. For the second
claim we proceed similarly except now in the last step we use \eqref{eq:eisenbd}, finding that
\begin{align*}
(\tfrac 14 + t^2)^{\ell} \langle E(\cdot, \tfrac 12 + it), \phi \rangle 
= \langle E(\cdot, \tfrac 12 + it), \Delta^{\ell} \phi \rangle 
\ll (CM)^{2\ell} (1+|t|) \int_{\mathcal F} \frac{dx dy}{y^{3/2}}.
\end{align*}
\end{proof}

We now derive estimates for $\langle |F_k|^2, \phi \rangle$ in two different ways.
The first approach, below, is obtained using the spectral decomposition, and follows Soundararajan's paper~\cite{S}, except for the use of Lemma \ref{lem:Iwaniec}.
\begin{lemma} \label{lem:spectral}
Let $\phi$ be as in Lemma \ref{lem:IBP}.
If $f$ is a Hecke cusp form of weight $k$ then
\begin{align*}
\langle |F_k|^2, \phi \rangle =&\frac{3}{\pi} \cdot  \langle 1 , \phi \rangle
+ O \Big ( (CM)^{3/2+\varepsilon} (\log k)^{\varepsilon} \cdot \Big ( \prod_{p \leq k} \Big ( 
1 - \frac{n(p)}{p} \Big ) +  \prod_{p \leq k} 
\Big ( 1 - \frac{\lambda_f(p^2)+1}{p} \Big ) \Big ) \| \phi \|_2 \Big )  \\ & + O_{A}( (\log k)^{-A})
\end{align*}
and where $n(p) = \lambda_f(p^2) + \tfrac 14 \cdot (1 - \lambda_f(p^2)^2)$.
\end{lemma}

\begin{proof}
Starting with the spectral decomposition we have (see for instance Theorem 15.5 of 
\cite{IK})
\begin{equation} \label{equ:mainque}
\langle |F_k|^2 , \phi \rangle =\frac{3}{\pi} \cdot \langle 1, \phi \rangle + \sum_{j \geq 1}
\langle |F_k|^2 , u_j \rangle \langle u_j , \phi \rangle + \frac{1}{4\pi}
\int_{\mathbb{R}} \langle |F_k|^2, E(\cdot, \tfrac 12 + it)\rangle \langle E
(\cdot , \tfrac 12 + it) , \phi \rangle dt. 
\end{equation}
By the previous lemma we have
$$
|\langle  u_j, \phi \rangle| \ll_{A} \frac{(CM)^{2A}}{1 + |t_j|^{2A}} \ \text{ and } \ 
|\langle E(\cdot, \tfrac 12 + it) , \phi \rangle| \ll_A  \frac{(CM)^{2A}}{1 + |t|^{2A-1}}
$$
for any fixed $A > 0$.  

Combining Corollary 1 of \cite{S}
with Lemma 2 in \cite{HS} we have
\begin{equation} \label{eq:soundbound}
|\langle |F_k|^2, E(\cdot, \tfrac 12 + it) \rangle|\ll (1 + |t|)
\exp \Big ( - \sum_{p \leq k} \frac{\lambda_f(p^2) + 1}{p} \Big ) (\log k)^{\varepsilon}.
\end{equation}
(Note that here we have used a slightly stronger form of Corollary 1 of \cite{S}, which
is easily seen to follow from the proof.)
Using the above bounds with Lemma \ref{lem:Iwaniec} it follows that the 
terms with $|t_j| > CM (\log k)^{\varepsilon}$  and $|t| > CM (\log k)^{\varepsilon}$ in (\ref{equ:mainque}) contribute an amount at most $O((\log k)^{-A})$. 
Recalling Weyl's law, that is $\sum_{|t_j| \le T} 1 \sim T^2/12$, which has been established
here by Selberg, and applying Lemma \ref{lem:Iwaniec} and Bessel's inequality it follows that
the contribution of the remaining cusp forms is bounded by
\begin{equation} \label{eq:specdiscbd}
\begin{split}
\Big ( \sum_{|t_j| \leq CM (\log k)^{\varepsilon}} |\langle |F_k|^2 , & u_j \rangle| ^2
\Big )^{1/2} \cdot \Big ( \sum_{j} |\langle u_j, \phi \rangle|^2 \Big )^{1/2}
 \\ 
& \ll (CM)^{3/2+\varepsilon} (\log k)^{\varepsilon} \cdot \prod_{p \leq k} \Big (1 - \frac{n(p)}{p}
\Big ) \cdot \| \phi \|_2.
\end{split}
\end{equation} 
The remaining Eisenstein series contribution is bounded by
\begin{equation} \label{eq:speccontbd}
\begin{split}
&\Big ( \int\limits_{|t| \leq CM (\log k)^{\varepsilon}} |\langle |F_k|^2 , E(\cdot, \tfrac 12 + it) \rangle| ^2 \, dt
\Big)^{1/2} \cdot \Big (  \int\limits_{\mathbb R} |\langle  E(\cdot, \tfrac 12 + it),  \phi \rangle| ^2 \, dt \Big)^{1/2} \\ 
&  \qquad \qquad \qquad \qquad \ll (CM)^{3/2} (\log k)^{\varepsilon} \cdot \prod_{p \leq k} \Big (1 - \frac{\lambda_f(p^2) + 1}{p}
\Big ) \cdot \| \phi \|_2
\end{split}
\end{equation} 
using (\ref{eq:soundbound}) and Bessel's inequality. 
Using \eqref{eq:specdiscbd} and \eqref{eq:speccontbd} in \eqref{equ:mainque} gives the claim.
\end{proof}

\subsection{Holowinsky's approach}
The next two lemmas combine to give an estimate for $\langle |F_k|^2, \phi \rangle$
by using a smoothed incomplete Eisenstein series and bounds for a shifted convolution problem. This mirrors the route taken by Holowinsky~\cite{H}. 
\begin{lemma} \label{lem:LSlem}
Let $h$ be a smooth, positive-valued function, such that $h^{(\ell)}(x) \ll M^{\ell}$ for all integers $\ell \geq 0$
and assume $M \le \log k$. 
Suppose in addition that $h$ is supported in $[1/2, C]$ with $C \le \log k$.  Let $I \subset [-1/2,1/2]$ be an interval. Then,
\begin{equation} \label{eq:mainHT}
\begin{split}
\int_{0}^{\infty} \int_{I} |F_k(z)|^2 h(y) \frac{dx dy}{y^2} =& \left(|I|+O\left( \frac{1}{(\log k)^3} \right) \right) \langle E(z|h) F_k, F_k \rangle \\
 & +O\left( (\log k)^{\varepsilon}  \prod_{p \le k} \left(1-\frac{(|\lambda_f(p)|-1)^2}{p} \right) \right),
\end{split}
\end{equation}
where $E(z|h)$ is the incomplete Eisenstein series given by
\[
E(z|h)=\sum_{\gamma \in \Gamma_{\infty} \backslash \Gamma} h(\tmop{Im}(\gamma z)).
\]
\end{lemma}

\begin{proof} Consider the following incomplete Poincare series, 
$$
P_{h,m}(z) := \sum_{\gamma \in \Gamma_{\infty} \backslash \Gamma}
h( \tmop{Im} (\gamma z)) e(m \tmop{Re} (\gamma z)),
$$
and note $P_{h,0}=E(\cdot|h)$.
 Using the standard unfolding method, we get
\begin{equation}\label{eq:unfold}
\langle |F_k|^2 , P_{h,m} \rangle = \int_{0}^{\infty} \int_{-1/2}^{1/2} 
|F_k(z)|^2   h(y) e(mx ) \, \frac{dx \, dy}{y^2}. 
\end{equation}
Applying Proposition 2.1 of \cite{LS1}, which follows from expanding $|F_k|^2$,
and keeping track of the dependencies on $m$ and $h$
one has
\begin{equation} \label{eq:LSprop}
\begin{split}
\langle |F_k|^2 , P_{h,m}(z) \rangle =& \frac{2\pi^2}{(k-1) L(1, \text{sym}^2 f)}
\sum_{r \ge 1} \lambda_f(r) \lambda_f(r + m)  h \Big ( \frac{k-1}{4\pi(r + m/2)}
\Big ) \\
&+ O\Big(\frac{(|m|+CM)^B}{k^{1-\varepsilon}}\Big), 
\end{split}
\end{equation}
where $B$ is a sufficiently large absolute constant.

Using Beurling-Selberg polynomials (see for instance Chapter 1 of Montgomery \cite{Montgomery}) there exists coefficients $a_{\ell,H}^{-}(I)$ and $ a_{\ell,H}^{+}(I)$ such that $|a_{\ell, H}^{\pm}(I)| \ll 1/\ell$ and
\begin{equation} \label{eq:BS}
|I| - \frac{1}{H+1} + \sum_{0 \neq |\ell| \leq H} a_{\ell,H}^{-}(I) e(\ell x)\leq 
\chi_{I}(x) \leq |I| + \frac{1}{H+1} + \sum_{0 \neq |\ell| \leq H} a_{\ell,H}^{+}(I)
e(\ell x) .
\end{equation}
Combining \eqref{eq:unfold}, \eqref{eq:LSprop}, and \eqref{eq:BS} it follows that
\begin{equation} \label{eq:evalQUE}
\begin{split}
&\int_{0}^{\infty} \int_{I} |F_k(z)|^2 h(y) \frac{dx dy}{y^2} =
\Big ( |I| + O \Big ( \frac{1}{H} \Big ) \Big ) \langle E(z|h) F_k, F_k \rangle \\
&  + O \Big ( \frac{1}{(k-1) L(1, \text{sym}^2 f)} \sum_{0 \neq |m| \leq H} \frac{1}{m} 
\sum_{r \ge 1} |\lambda_f(r) \lambda_f(r + m)| \cdot h \Big ( \frac{(k-1)}{4\pi (r + \frac{m}{2})} 
\Big )\Big) \\
& + O\Big(\frac{H(H+CM )^B}{k^{1-\varepsilon}}\Big ).
\end{split}
\end{equation}

To handle the off-diagonal terms in \eqref{eq:evalQUE}
we use a version of Shiu's bound (as in Holowinsky's work, see Theorem 1.2 of \cite{H}).
This gives
\begin{equation}\label{eq:evaloff}
\sum_{0 \neq |m| \leq H} \frac{1}{m} 
\sum_{r \ge 1} |\lambda_f(r) \lambda_f(r + m)| \cdot h \Big ( \frac{k-1}{4\pi (r + \frac{m}{2})} 
\Big ) \ll k (\log k)^\varepsilon (\log H)^2 \cdot \prod_{p \leq k} \Big ( 1 + \frac{2|\lambda_f(p)|-2}{p} \Big ).
\end{equation}
To complete the proof take $H=(\log k)^{3}$ and use \eqref{eq:evaloff} 
in \eqref{eq:evalQUE} along with the bound
$$
(\log \log k)^{-3} \exp \Big ( \sum_{p \leq k} \frac{\lambda_f(p^2)}{p} 
\Big ) \ll L(1 , \text{sym}^2 f) \ll (\log k)^2.
$$
(The lower bound in the above equation is proven in \cite{HS} while the upper bound is classical.) 
\end{proof}

In the next lemma we repeat the argument of Holowinsky to evaluate main term in \eqref{eq:mainHT}. 
 To do this we use the Fourier expansion of the incomplete Eisenstein series, which is given by
\[
E(z|h)=a_{0,h}(y)+\sum_{|\ell| \ge 1} a_{\ell,h}(y) e(\ell x).
\]
The Fourier coefficients are obtained from those of $E(z,s)$ (see equation \eqref{eq:eisen}). Writing $H$ for the Mellin transform of $h$ and noting $E(z|h)=\frac{1}{2\pi i} \int_{(2)} H(-s) E(z,s) \, ds$ one has, by shifting contours, for $\ell \neq 0$
\[
a_{\ell,h}(y)=\left(\frac{y}{\pi} \right)^{1/2} \int_{\mathbb R} \frac{\pi^{it} H(-\tfrac12-it)}{\Gamma(\tfrac12+it) \zeta(1+2it)} \tau_{it}(|\ell|) K_{it}(2\pi |\ell|y) \, dt.
\]
Observe that by repeatedly integrating by parts $H(-s) \ll \frac{ (CM)^A}{1+|t|^A}$, for any integer $A \ge 1$. 
Applying \eqref{eq:kbessel} it follows for any integer $A \ge 1$ that
\begin{equation} \label{eq:fouriercoef}
a_{\ell,h}(y) \ll  y^{1/2} d(|\ell|) \min\left( CM, \frac{(CM)^A}{|\ell y|^{A-2/3-\varepsilon}}\right).
\end{equation}
Additionally, we get by shifting contours that
\begin{equation} \label{eq:zerofourier}
\begin{split}
a_{0,h}(y)=&\frac{1}{2\pi i} \int_{(2)} H(-s)\left(y^s +\frac{\theta(1-s)}{\theta(s)}y^{1-s} \right) \, ds \\
=& \frac{3}{\pi}  H(-1)+O( (CM)^2 \sqrt{y}).
\end{split}
\end{equation}

\begin{lemma} \label{lem:holowinsky}
Let $h$ be as in the previous lemma. Then
\[
\begin{split}
\langle E(z|h) F_k, F_k \rangle=& \frac{3}{\pi} \int_0^{\infty} h(y) \, \frac{dy}{y^2} +O\left( (CM)^2 (\log k)^{\varepsilon} \prod_{p \le k} \left(1-\frac{\tfrac12(|\lambda_f(p)|-1)^2}{p} \right)\right).
\end{split}
\]
\end{lemma}

\begin{proof}

The proof closely follows the work of Holowinsky~\cite{H}, whose main analytic tool 
is
the smoothed incomplete Eisenstien series
\[
E^Y(z|g)=\sum_{\gamma \in \Gamma_{\infty}\backslash \Gamma} g(Y \tmop{Im}(\gamma z)),
\]
where $g$ is a fixed smooth function that is compactly supported on the positive reals.
Writing $G$ for the Mellin transform of $g$ and shifting contours it follows that
\[
\begin{split}
\langle E^Y(z|g) E(z|h) F_k, F_k \rangle=&
\frac{1}{2\pi i}\int_{(2)} G(-s) Y^s \langle E(z,s) E(z|h) F_k, F_k \rangle ds \\
=&Y \frac{3}{\pi} G(-1) \langle E(z|h) F_k, F_k \rangle \\
&+\frac{1}{2\pi i} \int_{(\frac12)} G(-s)Y^s \langle E(z,s) E(z|h) F_k, F_k \rangle ds.
\end{split}
\]
We bound the inner product in the last integral by applying \eqref{eq:eisenbd} to get by unfolding 
\[
\begin{split}
\langle E(z,s) E(z|h) F_k, F_k \rangle=&\int_0^{\infty} \int_{-1/2}^{1/2} h(y) E(z,s) |F_k(z)|^2 \frac{dx \, dy}{y^2} \\
\ll& (1+|s|) \int_{1/2}^{C} \int_{-1/2}^{1/2} h(y) \sqrt{y} \, |F_k(z)|^2 \frac{dx \, dy}{y^2} \ll \sqrt{C} (1+|s|) \lVert F_k \rVert_2.
\end{split}
\]
This gives
\begin{equation} \label{eq:contourshift}
\langle E^Y(z|g) E(z|h) F_k, F_k \rangle=Y \frac{3}{\pi} G(-1) \langle E(z|h) F_k, F_k \rangle+O\left( \sqrt{YC} \right).
\end{equation}

The Hecke cusp form $f$ has a Fourier expansion
\[
f(z)=\sum_{n \ge 1} a_f(n) e(nz).
\]
Since we have normalized with $\langle f, f\rangle=1$
the eigenvalues $\lambda_f(n)$ of the Hecke operators are related to the Fourier coefficients $a_f(n)$
by the relation
\[
\lambda_f(n) n^{(k-1)/2} a_f(1)=a_f(n)
\]
with 
\[
|a_f(1)|^2=\frac{2\pi^2 (4\pi)^{k-1}}{\Gamma(k) L(1, \tmop{sym}^2 f)}.
\]
We now use the unfolding method to get that
\begin{equation} \label{eq:fourier}
\begin{split}
\langle E^Y(z|g) E(z|h) F_k, F_k \rangle=&\int_0^{\infty} \int_{-1/2}^{1/2} g(Yy) E(z|h) |F_k(z)|^2 \frac{dx \, dy}{y^2} \\
=&  \frac{2\pi^2 (4\pi)^{k-1}}{\Gamma(k) L(1, \tmop{sym}^2 f)}
\sum_{ \ell  } \sum_{n \ge 1} \lambda_f(n) \lambda_f(n+\ell) (n(n+\ell))^{\frac{k-1}{2}} \\
& \qquad \qquad \times \int_0^{\infty} y^k g(Yy)  a_{\ell, h}(y) e^{-2\pi(2n+\ell)y} \frac{dy}{y^2}.
\end{split}
\end{equation}
Using Mellin inversion and Stirling's formula an argument of Luo and Sarnak \cite{LS1} gives
\begin{equation} \label{eq:gammafn}
\begin{split}
&\frac{(4\pi)^{k-1}}{\Gamma(k)}
(n(n+\ell))^{\frac{k-1}{2}}
 \int_0^{\infty} y^k g(Yy)  e^{-2\pi(2n+\ell)y} \frac{dy}{y^2} \\
&\qquad \qquad \qquad \qquad =\frac{1}{k-1} \cdot  g\left( \frac{Y(k-1)}{4\pi(n+\frac{\ell}{2})}\right)+O\left(\frac{1}{k^{1/2-\varepsilon}(n+\frac{\ell}{2})} \right),
\end{split}
\end{equation}
(see the proof of Proposition 2.1 of \cite{LS1} or the argument leading up to formula (20) of Holowinsky \cite{H}).

To bound the terms with $\ell \neq 0$ in \eqref{eq:fourier} we first use \eqref{eq:fouriercoef} and \eqref{eq:gammafn}.
Then 
we apply Shiu's bound as in the proof of the previous lemma. Thus, the terms with $\ell \neq 0$ are bounded by
\begin{equation} \label{eq:eisenoffdiag}
\begin{split}
\ll & \sum_{|\ell| \ge 1} \frac{ d(|\ell|) \min\left( CM, \frac{(MC)^2 }{|\ell Y^{-1}|^{4/3-\varepsilon}}\right) }{k \sqrt{Y} L(1,\tmop{sym}^2 f)} \sum_{n \ge 1} |\lambda_f(n)\lambda_f(n+\ell)| g\left(\frac{Y(k-1)}{4\pi(n+\frac{\ell}{2})} \right) \\
\ll & \frac{\sqrt{Y}}{L(1,\tmop{sym}^2 f)} \prod_{p \le k} \left( 1+\frac{2|\lambda_f(p)|-2}{p} \right)  \sum_{|\ell| \ge 1}\min\left( CM, \frac{(CM)^2}{|\ell Y^{-1}|^{4/3-\varepsilon}}\right)  d(|\ell|)^2 \\
\ll &  (\log k)^{\varepsilon} \cdot (CM)^{7/4} Y^{3/2+\varepsilon} \prod_{p \le k} \left( 1-\frac{\lambda_f(p^2)-2|\lambda_f(p)|+2}{p} \right).
\end{split}
\end{equation}

It remains to estimate the contribution from the zeroth Fourier coefficient of $E(z|h)$ in \eqref{eq:fourier}.
Assuming $Y \le \log k$ and using \eqref{eq:zerofourier} and \eqref{eq:gammafn}, the term with $\ell=0$ in the right-hand side of \eqref{eq:fourier} equals
\begin{equation} \label{eq:LSstep2}
\begin{split}
&\frac{2\pi^2}{(k-1) L(1, \tmop{sym}^2 f)} \sum_{n \ge 1} |\lambda_f(n)|^2 n^{k-1} \int_0^{\infty} g(Yy)  a_{0, h}(y) e^{-4\pi ny} \frac{dy}{y^2} \\
=& \left( \frac{3}{\pi}\langle E(z|h), 1 \rangle+O\left(\frac{(CM)^2}{\sqrt{Y}}\right) \right)\frac{2\pi^2}{(k-1) L(1, \tmop{sym}^2 f)} \sum_{n \ge 1} |\lambda_f(n)|^2 g\left( \frac{Y(k-1)}{4 \pi n} \right)
\end{split}
\end{equation}
To evaluate the sum on the right-hand side
 we employ Soundararajan's \cite{S} weak sub-convexity estimate. Let $G$ denote the Mellin transform of $g$ and observe that $G(s) \ll_A  (1 + |s|)^{-A}$ for any fixed $A$ in any vertical strip $-3 \le a \leq \tmop{Re} (s) \leq b<0$. Then
$$
\sum_{r \ge 1} |\lambda_f(r)|^2 \cdot g \Big ( \frac{Y(k-1)}{4\pi r} \Big )
= \frac{1}{2\pi i} \int_{(2)} \Big ( \frac{Y(k-1)}{4\pi} \Big )^{s}
\cdot \frac{L(s, f \otimes f)}{\zeta(2s)} \cdot G(-s) \ ds. 
$$
Shifting contours to $\text{Re} (s) = \tfrac12$ we collect a pole at $s = 1$
with residue 
\[
\frac{Y(k-1)}{4\pi} \cdot \frac{6}{\pi^2} G(-1) \cdot L(1, \text{sym}^2 f).
\]
To bound the integral
on the line $\text{Re} (s) = \tfrac12$ we use
the estimate
\[
|L(\tfrac12+it, \tmop{sym}^2 f) | \ll \frac{k^{1/2}(1+|t|)}{(\log k)^{1-\varepsilon}}
\]
due to Soundararajan \cite{S} (see Example (1.1)).
We conclude that
\begin{equation}\label{eq:evalmain}
\sum_{r \ge 1} |\lambda_f(r)|^2 \cdot g \Big ( \frac{Y(k-1)}{4\pi r} \Big )
= \frac{Y(k-1)}{4\pi} \cdot \frac{6}{\pi^2} G(-1) \cdot L(1, \text{sym}^2 f) + O \Big ( \frac{\sqrt{Y} \cdot k}{(\log k)^{1 - \varepsilon}} \Big ).
\end{equation}

Use the estimates \eqref{eq:eisenoffdiag}, \eqref{eq:LSstep2}, and \eqref{eq:evalmain} in \eqref{eq:fourier}. Next combine the resulting formula with \eqref{eq:contourshift}. Finally, use the bound $L(1, \tmop{sym}^2 f) \gg (\log k)^{-1}$ (which follows from the work of Hoffstein and Lockhart \cite{HL}) to get
\[
\begin{split}
\langle E(z|h) F_k, F_k \rangle=& \frac{3}{\pi} \langle E(z|h) ,1 \rangle +O\left( \frac{ (CM)^2 (\log k)^{\varepsilon}}{\sqrt{Y}} \right)\\
&+O\left( (\log k)^{\varepsilon}(CM)^{7/4} \sqrt{Y} \prod_{p \le k} \left(1-\frac{(|\lambda_f(p)|-1)^2}{p} \right) \right).
\end{split}
\]
To complete the proof take 
$$Y=\prod_{p \le k} \left(1+\frac{(|\lambda_f(p)|-1)^2}{p} \right).$$

\end{proof}

\subsection{Proof of Effective QUE}
\begin{proof}[Proof of Theorem~\ref{effective que}]
Let $\mathcal{R} = \{ (x,y): x\in I, y \in J\} \subset \mathcal F$ be a rectangular domain
and write $\mathcal R=I\times J$ where
 $I$ and $J$ are intervals. 
Let $\mathcal R'=\mathcal R \cap \mathcal \{z \in \mathcal F:
\tmop{Im}(z) \le (\log k)^{\eta_1}\}$, where $0< \eta_1 \le 1$
will be chosen later,
and note that $\mathcal R'$ is also a rectangular domain. 
We now will use a result of Soundararajan which
bounds the amount of $L^2$-mass of $y^{k/2}f(z)$
high in the cusp. This enables
us to restrict
to rectangular regions of the form $\mathcal R'$.
From the main result of Soundararajan 
\cite{Soundmass} we have that
\[
\iint\limits_{\substack{|\tmop{Re}(z)|\le \frac12 
\\ \tmop{Im}(z) \ge (\log k)^{\eta_1}}} y^k |f(z)|^2 \frac{dx \, dy}{y^2}
\ll \frac{1}{(\log k)^{\eta_1/2-\varepsilon}}.
\]
Thus,
\begin{equation} \label{eq:QUErestrict}
\Big|\langle |F_k|^2, \chi_{\mathcal R} \rangle
-\frac{3}{\pi} \tmop{Area}_{\mathbb H}(\mathcal R) \Big|
= \Big|\langle |F_k|^2, \chi_{\mathcal R'} \rangle-\frac{3}{\pi} \tmop{Area}_{\mathbb H}(\mathcal R') \Big|+
O\Big(\frac{1}{(\log k)^{\eta_1/2-\varepsilon}} \Big).
\end{equation} 
Hence, we may restrict our attention to rectangular domains
lying inside $\{z \in \mathcal F:
\tmop{Im}(z) \le (\log k)^{\eta_1}\}$ at the cost
of an error that is $O((\log k)^{-\eta_1/2+\varepsilon})$.

 We now consider smooth functions $\phi_{J}^{\pm}(y)$ that majorize or minorize (resp.) the
characteristic function of
the interval $J=[c,d]$, where $d\le C \le (\log k)^{\eta_1}$.
Let $\phi_{J}^{\pm}(y)$ be such that $\phi_{J}^{\pm}(y) = 1$ for $y \in J$.
Moreover, 
suppose that $\phi_{J}^{\pm}(y)$ is supported in $J_{\delta}=[c\mp\delta,d\pm\delta]$ 
and satisfies $(\phi_{J}^{\pm})^{(\ell)}(y) \ll (1/\delta)^{\ell}$ for all $\ell \geq 1$. We also pick a $\phi_{I}^{\pm}(x)$ with identical properties. Consider $\phi^{\pm}(x,y) = \phi_{I}^{\pm}(x) \phi_{J}^{\pm}(x)$. Then we easily see that $\Delta^{\ell} \phi^{\pm} \ll (1/\delta)^{2\ell}$. We also choose
 $\delta=(\log k)^{-\eta_2}$, with $0<\eta_2 \le 1$
to be chosen later.

Applying Lemmas \ref{lem:LSlem} and \ref{lem:holowinsky} twice with $h(y)=\phi_J^{\pm}(y)$ we have,
\begin{equation} \label{eq:evalQUEMain1}
\begin{split}
\langle |F_k|^2 , \chi_{\mathcal{R}} \rangle =& \frac{3}{\pi} \cdot \langle 1, \chi_{\mathcal{R}} \rangle + O(\delta) + O \Big (  (C/\delta)^2 (\log k)^{\varepsilon} \prod_{p \leq k}
\Big ( 1 - \frac{\tfrac12 (|\lambda_f(p)| - 1)^2}{p} \Big ) \Big ).
\end{split}
\end{equation} 
On the other hand, according to Lemma \ref{lem:spectral} with $\phi(x,y)=\phi^{\pm}(x,y)$
\begin{equation} \label{eq:evalQUEMain2}
\begin{split}
\langle |F_k|^2,  \chi_{\mathcal{R}} \rangle =& \frac{3}{\pi} \cdot \langle 1, \chi_{R} \rangle 
+ O(\delta) +  O_A((\log k)^{-A})  \\ 
&+  O \Big ( (C/\delta)^{3/2+\varepsilon} \cdot (\log k)^{\varepsilon} \Big ( \prod_{p \leq k} \Big ( 1 - \frac{n(p)}{p} \Big ) + \prod_{p \leq k}
\Big ( 1 - \frac{\lambda_f(p^2) + 1}{p} \Big ) \Big ) .
\end{split}
\end{equation} 

First we set $\eta_2=\tfrac{1}{2} \cdot \eta_1$ to balance the error terms
in \eqref{eq:QUErestrict}, \eqref{eq:evalQUEMain1}, and \eqref{eq:evalQUEMain2} that do not contain Euler products.
To balance the error terms with Euler products it remains to
optimize
\[
\min \Big ( \prod_{p \leq k} \Big ( 1 - \frac{\tfrac12(|\lambda_f(p)| - 1)^2}{p} \Big )
\ , \ \prod_{p \leq k} \Big ( 1 - \frac{n(p)}{p} \Big ) + \prod_{p \leq k}
\Big ( 1 - \frac{\lambda_f(p^2) + 1}{p} \Big ) \Big ).
\]
For $a,b,c \geq 0$ we have
$$
\min(a , b +c ) \leq \min(a,b) + \min(a,c) \ll a^{\alpha} b^{1 - \alpha}
+ a^{\beta} c^{1 - \beta}.
$$
Therefore it is enough to choose $\alpha$ and $\beta$ so as to minimize separately $a^{\alpha} c^{1 - \alpha}$ and $b^{\beta} c^{1 - \beta}$ for $a,b,c$ corresponding to the Euler products above. 
To shorten notation write $\lambda=|\lambda_f(p)|$.
This leads us to finding an $0 \leq \alpha \leq 1$ which minimizes
$$\max_{0 \leq \lambda \leq 2} \Big ( - \tfrac{\alpha}{2} (\lambda - 1)^2 - (1 - \alpha) (\lambda^2 - 1 - \tfrac 14 (\lambda^2 - 1)^2 + \tfrac 14) \Big ).$$
We also need to find a $0 \leq \beta \leq 1$ which will minimize 
$$
\max_{0 \leq \lambda \leq 2} \Big ( - \tfrac{\beta}{2} (\lambda - 1)^2  - (1 - \beta) \lambda^2 \Big ).
$$
This is minimized by taking $\beta = 2-\sqrt{2}$ and under this choice the maximum is less than $-\tfrac {1}{12}$. For the first condition, let us first restrict to $\alpha \geq 1/3$.
 We note that we can then restrict to $\lambda \leq 1$, because for $\lambda \geq 1$ the $\max$ is always bounded by $- \tfrac {1}{12}$. In the range $0 \leq \lambda \leq 1$, we have $\tfrac 14 (\lambda^2 - 1)^2 \leq \tfrac14 (\lambda - 1)^2$. Thus it's enough to optimize
$$
\max_{0\leq \lambda \leq 1} 
\Big ( - \frac{\alpha}{2} (\lambda - 1)^2 - (1-\alpha) (\lambda^2 -1 - \tfrac 14 (\lambda - 1)^2 + \tfrac 14) \Big ).
$$
For $\tfrac13 \leq \alpha \leq 1$ this maximum is equal to
$$
\frac{(1-\alpha)(13-15\alpha)}{4(3-\alpha)}.
$$
This is smallest when $\alpha = 3-8/\sqrt{15}$ and
the minimum is then
$$
-\kappa := -31/2+4\sqrt{15} = -0.008066615 \ldots.
$$
Thus, the minimum of the error terms in \eqref{eq:evalQUEMain1} and \eqref{eq:evalQUEMain2}
with Euler products is $\ll (\log k)^{2\eta_1+2\eta_2-\kappa+\varepsilon}$.
Since we chose $\eta_2= \tfrac{1}{2} \cdot \eta_1$, we obtain
\[
\langle |F_k|^2,  \chi_{\mathcal{R}} \rangle = \frac{3}{\pi} \cdot \langle 1, \chi_{R} \rangle + O((\log k)^{-\eta_1/2+\varepsilon}) + O((\log k)^{3\eta_1-\kappa+\varepsilon}).
\]
So this balances by taking $\eta_1=2/7 \cdot \kappa$ and gives an error of $O((\log k)^{-\kappa/7})$ as claimed.
\end{proof}

\textbf{Acknowledgments.} The third author is grateful to Zeev Rudnick for inviting him to Tel-Aviv University where this work was started. We would also like to thank Misha Sodin and Zeev Rudnick for many stimulating and helpful discussions.
Additionally, we are also grateful to Roman Holowinsky for sending us lecture notes from Henryk Iwaniec's course on QUE. We would also like to thank Matt Young for comments on an earlier draft of this paper and in particular for pointing out a better proof of Lemma \ref{lem:IBP}.

\bibliographystyle{amsplain}
\bibliography{paper}

\end{document}